\documentclass{amsart}

\usepackage{amsmath, amsthm, amssymb, amstext, amsfonts}
\usepackage{enumerate}
\usepackage{graphicx}
\usepackage[T1]{fontenc} 
\usepackage[applemac]{inputenc}

\theoremstyle{plain}

\newtheorem{thm}{Theorem}[section]
\newtheorem{lem}[thm]{Lemma}
\newtheorem{cor}[thm]{Corollary}
\newtheorem{prop}[thm]{Proposition}

\newtheorem{prob}[thm]{Problem}
\newtheorem*{prob*}{Problem}

\theoremstyle{definition}

\newtheorem{ex}[thm]{Example}

\newcommand{\R}{\ensuremath{\mathbb{R}}}
\newcommand{\N}{\ensuremath{\mathbb{N}}}
\newcommand{\Z}{\ensuremath{\mathbb{Z}}}

\newcommand{\cB}{\ensuremath{\mathcal{B}}}
\newcommand{\cO}{\ensuremath{\mathcal{O}}}
\newcommand{\cP}{\ensuremath{\mathcal{P}}}
\newcommand{\cR}{\ensuremath{\mathcal{R}}}

\newcommand{\inv}{\ensuremath{^{-1}}}

\newcommand{\rand}{\partial}

\newcommand{\es}{\ensuremath{\emptyset}}

\newcommand{\sub}{\subseteq}

\def\qi{quasi-iso\-metric}
\def\qiy{quasi-iso\-me\-try}
\def\qis{quasi-iso\-me\-tries}

\def\lf{locally finite}

\def\qg{quasi-geo\-desic}

\newcommand{\comment}[1]{}

\newcommand{\nat}{{\mathbb N}}

\newcommand{\lrarrow}{{\leftrightarrow}}

\def\hm{semimetric}
\def\Hm{Semimetric}

\def\?#1{\vadjust{\vbox to 0pt{\vss\vskip-8pt\leftline{%
     \llap{\hbox{\vbox{\pretolerance=-1
     \doublehyphendemerits=0\finalhyphendemerits=0
     \hsize16truemm\tolerance=10000\small
     \lineskip=0pt\lineskiplimit=0pt
     \rightskip=0pt plus16truemm\baselineskip8pt\noindent
     \hskip0pt        
     #1\endgraf}\hskip7truemm}}}\vss}}}

\newenvironment{txteq*}
  {
    \begin{equation*}
    \begin{minipage}[c]{0.9\textwidth} 
    \em                                
  }
  {\end{minipage}\end{equation*}\ignorespacesafterend}
\newenvironment{txteqtag}[1]
  {
    \begin{equation}\tag{#1}
    \begin{minipage}[c]{0.85\textwidth} 
    \em                                
  }
  {\end{minipage}\end{equation}\ignorespacesafterend}

\begin{document}

\title{A boundary for hyperbolic digraphs and semigroups}
\author{Matthias Hamann}
\thanks{Funded by the German Research Foundation (DFG) -- project number 448831303.}
\address{Matthias Hamann, Mathematics Institute, University of Warwick, Coventry, UK}
\date{}

\begin{abstract}
Based on a notion by Gray and Kambites of hyperbolicity in the setting of \hm\ spaces like digraphs or semigroups, we will construct (under a small additional geometric assumption) a boundary based on quasi-geodesic rays and anti-rays that is preserved by \qis\ and, in the case of locally finite digraphs and right cancellative semigroups, refines their ends.
Among other results, we show that it is possible to equip the space, if it is finitely based, together with its boundary with a pseudo-\hm.
\end{abstract}


\maketitle

\section{Introduction}\label{sec_Intro}

Gray and Kambites~\cite{GK-HyperbolicDigraphsMonoids} gave a geometric notion for hyperbolic semigroups, or more generally hyperbolic digraphs or hyperbolic \hm\ spaces, that generalises Gromov's notion for groups~\cite{gromov}.
The important difference to the undirected situation is that the distance function is not symmetric and this is taken into account for Gray and Kambites' definition, see Section~\ref{sec_thin}.

In~\cite{H-HyperbolicDigraph}, we have shown that their notion of hyperbolicity is a \qiy\ invariant, if we ask the spaces to satisfy an additional geometric assumption.
This assumption bounds the lengths of geodesics in in- and out-balls of fixed radius, similar as in the case of metric spaces, where geodesics within balls of radius $r$ have length at most~$2r$.
See Section~\ref{sec_thin} for more on this property.

In this paper, we will look at a boundary for hyperbolic spaces that will satisfy the above mentioned geometric assumption, so we can also use that hyperbolicity for those spaces is preserved by \qis.

In hyperbolic metric spaces, the hyperbolic boundary can be defined by an equivalence relation on geodesic rays, where two such rays are equivalent if they are eventually close to each other.
For hyperbolic \hm\ spaces, this is no longer true, but the above relation is still a quasiorder on the geodesic rays and anti-rays, see Section~\ref{sec_Boundary}.
Quasiorders canonically gives rise to an equivalence relation.
In our situation, the corresponding equivalence classes will be our geodesic boundary points.

We can define the same relation on  \qg\ rays and anti-rays and the boundary points defined by that relation are trivially preserved by \qis.
By $\rand X$ we denote the \qg\ boundary of a hyperbolic \hm\ space~$X$.
For the geodesic boundary, however, it is still unknown whether they are preserved by \qis.
While in the case of proper hyperbolic geodesic metric spaces, we can apply the Arzel\`a-Ascoli theorem to prove that both boundaries are essentially the same, it is not possible to apply an analogue theorem in the case of \hm\ spaces, since that is false, in general, see~\cite{CZ-Arzela-Ascoli}.
But for locally finite digraphs, we can use a compactness argument to show that the \qg\ boundary and the geodesic boundary coincide and thus conclude that the geodesic boundary is preserved by \qis, see Section~\ref{sec_qgBound}.

For general digraphs, there is another notion for a boundary: ends as defined by Zuther~\cite{Z-EndsDigraphs}.
It turns out that the geodesic boundary is a refinements of the end space for locally finite digraphs.
Furthermore, we prove that the ends of locally finite digraphs are also preserved by \qis, see Section~\ref{sec_Ends}.

\Hm\ spaces~$X$ have two natural topologies associated to them: one has the open out-balls of finite radius as base and the other the open in-balls of finite radius.
These two topologies extend to topologies of $X\cup\rand X$ and \qis\ between hyperbolic \hm\ spaces extend to homeomorphisms on the spaces with their geodesic boundaries with respect to these two topologies (see Section~\ref{sec_topology}).
Furthermore, if~$X$ has a finite base, we can equip $X\cup\rand X$ with a pseudo-\hm, whose induced topologies coincide with the just mentioned ones under small additional assumptions, see Section~\ref{sec_BoundaryHM}.

Further results regarding the geodesic boundary include that, for locally finite digraphs $D$, the space $D\cup\rand D$ is f-complete and b-complete (see Section~\ref{sec_PropBound}), two notions that mimic completeness in the setting of \hm s with respect to the two topologies that we discussed.
Additionally, we prove a result that can be seen as a partial analogue of the fact that in the case of metric spaces the ends correspond to the connected components of the hyperbolic boundary.
We use that result to obtain some results on the size of the geodesic boundary, see Section~\ref{sec_size}.

In the final section, Section~\ref{sec_semigroups}, we apply our results to semigroups.
Our additional geometric assumption implies that the results only hold for right cancellative semigroups, so we have defined the geodesic boundary for hyperbolic semigroups, too, in that we define it for a locally finite hyperbolic Cayley digraph of that semigroup.

As the geodesic boundary of hyperbolic digraphs or semigroups is preserved by \qis\ and thus by changing the finite generating set of finitely generated right cancellative semigroups, the geodesic boundary gives rise to a boundary of semigroups not just for one particular generating set.
The results on the number of geodesic boundary points implies for finitely generated cancellative hyperbolic semigroups that they have either $0$, $1$, $2$ or infinitely many geodesic boundary points.
Moreover, if the semigroup has exactly one end, then is has either $1$ or infinitely many geodesic boundary points.

We end by some discussions about finitely generated right cancellative hyperbolic semigroups with at most two geodesic boundary points.

\section{Preliminaries}

In this section, we will define all the basic notions for \hm\ spaces and digraphs.

\subsection{\Hm\ spaces}\label{sec_HMSpaces}

A map $d\colon X\times X\to[0,\infty]$ on a set~$X$ is a \emph{pseudo-\hm} if
\begin{enumerate}[(i)]
\item $d(x,x)=0$ for all $x\in X$ and
\item $d(x,y)\leq d(x,z)+d(z,y)$ for all $x,y,z\in X$,
\end{enumerate}
and we call $(X,d)$ a \emph{pseudo-\hm\ space}.
A pseudo-\hm\ is a \emph{\hm} if the following holds
\begin{enumerate}
\item[(i')] $d(x,y)=0$ if and only if $x=y$ for all $x,y\in X$,
\end{enumerate}
and we call $(X,d)$ a \emph{\hm\ space}.

If $X$ is a (pseudo-)\hm\ space and $x,y\in X$, then we set
\[
d^\lrarrow(x,y):=\min \{d(x,y),d(y,x)\}.
\]

Whereas metrics naturally define a topology based on open balls with respect to the metric, (pseudo-)\hm\ spaces $X$ come along with two natural topologies.
One is defined via the out-balls and the other via the in-balls:
for $r\geq 0$ and $x\in X$, we set the \emph{out-ball} and the \emph{open out-ball} of radius~$r$ around~$x$ as
\[
\cB^+_r(x):=\{y\in X\mid d(x,y)\leq r\}, \qquad \mathring{\cB}^+_r(x):=\{y\in X\mid d(x,y)< r\}
\]
and the \emph{in-ball} and the \emph{open in-ball} of radius~$r$ around~$x$ as
\[
\cB^-_r(x):=\{y\in X\mid d(y,x)\leq r\}, \qquad \mathring{\cB}^-_r(x):=\{y\in X\mid d(y,x)< r\},
\]
respectively,
and the \emph{forward} topology $\cO_f$ is generated by the open out-balls $\mathring{\cB}^+_r(x)$ for all $r\geq 0$ and $x\in X$ and the \emph{backward} topology $\cO_b$ is generated by the open in-balls $\mathring{\cB}^-_r(x)$ for all $r\geq 0$ and $x\in X$.

For $a,b\in\R$ with $a<b$, a map $P\colon [a,b]\to X$ is a \emph{directed path} if it is continuous with respect to the forward topologies and with respect to the backward topologies.
We call $P(a)$ the \emph{starting point} and $P(b)$ the \emph{end point} of~$P$ and we define the \emph{length} $\ell(P)$ of~$P$ as
\[
\ell(P):=\lim_{N\to\infty}\sum_{i=1}^N d(P(t_{i-1}),P(t_i))
\]
with $t_i:=a+i(b-a)/N$ for all $0\leq i\leq N$.
For $x,y\in X$, a \emph{directed $x$-$y$ path} is a directed path with starting point~$x$ and end point~$y$ and, for $U,V\sub X$, a \emph{directed $U$-$V$ path} is a directed path with starting point in~$U$ and end point in~$V$.

Two directed paths $P\colon [a,b]\to X$ and $Q\colon [a',b']\to X$ are \emph{parallel} if $P(a)=Q(a')$ and $P(b)=Q(b')$ and they are \emph{composable} if $P(b)=Q(a')$.
If they are composable, then we denote their \emph{composition} by~$PQ$.

A point $x\in X$ \emph{lies on $P$} if it lies in the image of~$P$.
For $x,y$ on~$P$ with $x=P(r)$ and $y=P(R)$ such that $r<R$, we denote by $xPy$ the subpath of~$P$ with starting point $x$ and end point~$y$.
For $r\geq 0$, we denote by $\cB^+_r(P)$ and by $\cB^-_r(P)$ the out-ball and the in-ball of radius $r$ around the image of~$P$, respectively.

Let $(X,d_X)$ and $(Y,d_Y)$ be two \hm\ spaces.
Let $\gamma\geq 1$ and $c\geq 0$.
A map $f\colon X\to Y$ is a \emph{$(\gamma,c)$-\qi\ embedding} if
\[
\gamma\inv d_X(x,x')-c\leq d_Y(f(x),f(x'))\leq\gamma d_X(x,x')+c
\]
for all $x,x'\in X$.
It is a \emph{$(\gamma,c)$-\qiy} if additionally for every $x\in X$ there is $y\in Y$ such that $d(f(x),y)\leq c$ and $d(y,f(x))\leq c$ and we say that $X$ is \emph{\qi} to~$Y$.
If the particular constant are not important, we simply talk about \emph{\qis} or \emph{\qi\ embeddings}.
Gray and Kambites~\cite[Proposition 1]{GK-SemimetricSpaces} remarked that being \qi\ is an equivalence relation.
An \emph{isometry} is a $(1,0)$-\qiy.

For $x,y\in X$, an \emph{$x$-$y$ geodesic} or a \emph{geodesic} from~$x$ to~$y$ is a directed path $P$ from~$x$ to~$y$ with $\ell(P)=d(x,y)$ and, for $U,V\sub X$, a \emph{$U$-$V$ geodesic} is a directed $U$-$V$ path that is a geodesic.
For $\gamma\geq 1$ and $c\geq 0$, a \emph{$(\gamma,c)$-\qg} from~$x$ to~$y$ is a directed path from~$x$ to~$y$ with
\[
\ell(uPv)\leq \gamma d(u,v)+c
\]
for all $u,v$ on~$P$ with $v$ on~$uPy$.
The \hm\ space $X$ is \emph{geodesic} if there exists an $x$-$y$ geodesic for all $x,y\in X$ with $d(x,y)<\infty$.

Let $(x_i)_{i\in\N}$ be a sequence in~$X$.
It \emph{f-converges} to $x\in X$ if it converges to~$x$ with respect to~$\cO_f$ and it \emph{b-converges} to~$x$ if it converges to~$x$ with respect to~$\cO_b$.
The sequence is called \emph{forward Cauchy}, or \emph{f-Cauchy}, if for every $\varepsilon>0$ there exists some $N\in\N$ such that $d(x_n,x_m)<\varepsilon$ for all $m\geq n\geq N$.
It is called \emph{backward Cauchy}, or \emph{b-Cauchy}, if for every $\varepsilon>0$ there exists some $N\in\N$ such that $d(x_m,x_n)<\varepsilon$ for all $m\geq n\geq N$.
We call $X$ \emph{f-complete} if every f-Cauchy sequence b-converges to a point in~$X$ and we call $X$ \emph{b-complete} if every b-Cauchy sequence f-converges to a point in~$X$.
Note that there are different notions of completeness for \hm\ spaces, see e.\,g.\ \cite{CZ-Arzela-Ascoli,HOA-QuasiMetricSpaces,SMW-QuasiMetricSpaces}; some of them differ from ours, e.\,g.\ in that they ask for f-completeness that f-Cauchy sequences f-converge.

We call $X$ \emph{sequentially f-compact} if for every sequence $(x_i)_{i\in\N}$ in~$X$ that satisfies $d(x_i,x_j)<\infty$ for all $i<j$ has a b-convergent subsequence.
We call $X$ \emph{sequentially b-compact} if for every sequence $(x_i)_{i\in\N}$ in~$X$ that satisfies $d(x_j,x_i)<\infty$ for all $i<j$ has an f-convergent subsequence.

\begin{prop}\label{prop_CompactComplete}
Let $X$ be a \hm\ space.
Then the following hold.
\begin{enumerate}[\rm (i)]
\item\label{itm_CompactComplete_1} If $X$ is sequentially f-compact, then it is f-complete.
\item\label{itm_CompactComplete_2} If $X$ is sequentially b-compact, then it is b-complete.
\end{enumerate}
\end{prop}

\begin{proof}
Let $X$ be sequentially f-compact.
Let $(x_i)_{i\in\N}$ be an f-Cauchy sequence in~$X$.
Then there is a subsequence $(x_{n_k})_{k\in\N}$ that b-converges to a point $x\in X$.
For $\varepsilon>0$, there exists $N\in\N$ such that $d(x_i,x_j)<\varepsilon/2$ for all $j>i>N$ and, for every $i>N$, there exists $k(i)\in\N$ with $d(x_{n_{k(i)}},x)<\varepsilon/2$.
Thus, we have
\[
d(x_i,x)<d(x_i,x_{n_{k(i)}})+d(x_{n_{k(i)}},x)<\varepsilon
\]
for all $i>N$.
So $(x_i)_{i\in\N}$ b-converges to~$x$ and $X$ is f-complete.

We obtain (\ref{itm_CompactComplete_2}) by an analogous argument.
\end{proof}

\subsection{Digraphs}\label{sec_digraphs}

A digraph $D=(V(D),E(D))$ is pair of a vertex set $V(D)$ and an edge set $E(D)$ such that it is an orientation of a multigraph.
That means, we are allowed to have loops and edges between the same vertices but reversely oriented and we are also allowed to have parallel edges in the same direction.
Whereas the latter is unimportant for most of our arguments, it play a major role, when we want to apply our result for semigroups.
For $U\sub V(D)$, we denote by $D[U]$ the digraph \emph{induced by} $U$, i.\,e.\ the digraph with vertex set $U$ and all edges of~$D$ both of whose incident vertices lie in~$U$.

A \emph{directed path} is a sequence $x_0\ldots x_n$ of vertices such that $x_ix_{i+1}\in E(D)$ for all $0\leq i<n$ and a \emph{proper directed path} is a sequence $x_0\ldots x_n$ of pairwise distinct vertices such that $x_ix_{i+1}\in E(D)$ for all $0\leq i<n$.
The \emph{length} $\ell(P)$ of a directed path~$P$ is the number of edges of the path.

If $x_0,x_1,\ldots$ are distinct vertices in~$D$ with $x_ix_{i+1}\in E(D)$ we call $x_0x_1\ldots$ a \emph{ray}, so it is a one-way infinite directed path with starting vertex~$x_0$.
If we have $x_{i+1}x_i\in E(D)$ instead we say that $x_0x_1\ldots$ is an \emph{anti-ray} and we have a one-way infinite directed path with end vertex~$x_0$.

For $x,y\in V(D)$, the \emph{distance} from~$x$ to~$y$, denoted by $d(x,y)$, is the length of a shortest directed path from~$x$ to~$y$ or, if no such path exists, then it is~$\infty$.
Whereas for graphs, this distance function is a metric, in the case of digraphs it is only a \hm.
The \emph{out-degree} of $x\in V(D)$ is the number of $y\in V(D)$ with $d(x,y)=1$ and its \emph{in-degree} is the number of $y\in V(D)$ with $d(y,x)=1$.
A digraph is \emph{\lf} if all of its in- and out-degrees are finite.

\section{Thin triangles}\label{sec_thin}

Let $X$ be a geodesic \hm\ space.
A \emph{triangle} consists of three points of~$X$ and three directed paths, one between every two of those points.
These paths are the \emph{sides} and the three point are the \emph{end points} of the triangle.
We call the triangle is \emph{geodesic} if all three sides are geodesics and we call it \emph{transitive} if two of its sides are composable and their composition is parallel to the third side.

Let $\delta\geq 0$.
We call a geodesic triangle with sides $P,Q,R$ \emph{$\delta$-thin} if the following holds:
\begin{txteq*}
if the starting point of~$P$ is either the starting or the end point of~$Q$ and the last point of~$P$ is either the starting or the end point of~$R$, then $P$ is contained in $\cB^+_\delta(Q)\cup \cB^-_\delta(R)$.
\end{txteq*}
If all geodesic triangles in~$X$ are $\delta$-thin then we call $X$ \emph{$\delta$-hyperbolic}.
If the constant $\delta$ in not important for us, we simply call $X$ \emph{hyperbolic}.

As mentioned in the introduction, in a \hm\ space $X$, the lengths of geodesics with starting and end point in $\cB^+_r(x)$ for $x\in X$ and $r\in R$ need not be bounded.
(Contrary to that, in metric spaces, this is always bounded by~$2r$.)
However, we will generally restrict ourselves to situations, where this is satisfied.
For that, we define the following two properties.

\begin{txteqtag}{B1}\label{itm_Bounded1}
There exists a function $f\colon \R\to\R$ such that for every $x\in X$, for every $r\geq 0$ and for all $y,z\in \cB^+_r(x)$ the distance $d(y,z)$ is either $\infty$ or bounded by $f(r)$.
\end{txteqtag}
\begin{txteqtag}{B2}\label{itm_Bounded2}
There exists a function $f\colon \R\to\R$ such that for every $x\in X$, for every $r\geq 0$ and for all $y,z\in \cB^-_r(x)$ the distance $d(y,z)$ is either $\infty$ or bounded by $f(r)$.
\end{txteqtag}

Besides in the case of metric spaces, where these properties are satisfied for the function $f(r)=2r$, they are also satisfied in hyperbolic digraphs of bounded in- and bounded out-degree by \cite[Lemma 3.2]{H-HyperbolicDigraph}.

The following two results are from~\cite{H-HyperbolicDigraph} and play major roles in calculating distances in hyperbolic \hm\ spaces.

\begin{prop}\label{prop_reverseDistanceShortDelta1}\cite[Proposition 3.3]{H-HyperbolicDigraph}
Let $\delta\geq 0$ and let $X$ be a $\delta$-hyperbolic geodesic \hm\ space that satisfies (\ref{itm_Bounded1}) for the function $f\colon \R\to\R$ and (\ref{itm_Bounded2}) for the function~$g\colon \R\to\R$.
\begin{enumerate}[\rm (i)]
\item\label{itm_reverseDistanceShortDelta1_1} If $P,Q,R$ are the sides of a geodesic triangle such that the starting point of~$P$ is either the starting or the end point of~$Q$ and the end point of~$P$ is either the starting or the end point of~$R$, then we have
\[
\ell(P)\leq (\ell(Q)/\varepsilon)f(\delta+\varepsilon)+(\ell(R)/\varepsilon)g(\delta+\varepsilon)
\]
for all $\varepsilon>0$.
\item\label{itm_reverseDistanceShortDelta1_2} If $x,y \in X$ with $d(x,y)\neq\infty$ and $d(y,x)\neq\infty$, then we have
\[
d(x,y)\leq (d(y,x)/\varepsilon) f(\delta+\varepsilon) + g(\delta)
\]
and
\[
d(x,y)\leq (d(y,x)/\varepsilon) g(\delta+\varepsilon)+f(\delta)
\]
for all $\varepsilon>0$.
\end{enumerate}
\end{prop}

\begin{lem}\label{lem_parallelSideCloseToAndFrom}\cite[Lemma 3.4]{H-HyperbolicDigraph}
Let $\delta\geq0$ and let $X$ be a $\delta$-hyperbolic geodesic \hm\ space that satisfies (\ref{itm_Bounded1}) and (\ref{itm_Bounded2}) for the function $f\colon\R\to\R$.
Let $P,Q,R$ be the sides of a geodesic triangle such that $P$ and~$Q$ are composable and their composition is parallel to~$R$.
Then $R$ lies in the out-ball of radius $6\delta+2\delta f(\delta+1)$ around $P\cup Q$ and in the in-ball of the same radius around $P\cup Q$.
\end{lem}

The concept of geodesic stability from~\cite{H-HyperbolicDigraph} will be used a couple of times in this paper, too.
So let us briefly define it and state the main result that we will apply.

Let $X$ be a geodesic \hm\ space.
We say that $X$ satisfies \emph{geodesic stability} if for all $\gamma\geq 1$ and $c\geq 0$ there exists a $\kappa\geq 0$ such that, for all $x,y\in X$ and all $(\gamma,c)$-\qg s $P$ and~$Q$ from~$x$ to~$y$, every point of~$P$ lies in $\cB^+_\kappa(Q)\cap \cB^-_\kappa(Q)$.

\begin{thm}\label{thm_Delta1GeodStab}\cite[Corollary 6.3]{H-HyperbolicDigraph}
Every hyperbolic geodesic \hm\ space that satisfies (\ref{itm_Bounded1}) and (\ref{itm_Bounded2}) satisfies geodesic stability.
\end{thm}

Another result that we will use is the main result of~\cite{H-HyperbolicDigraph}, which says that hyperbolicity is a property that is preserved by \qis, if the geodesic \hm\ spaces satisfy the properties (\ref{itm_Bounded1}) and~(\ref{itm_Bounded2}).

\begin{prop}\label{prop_QIPreserveHyp}\cite[Proposition 7.2]{H-HyperbolicDigraph}
Let $X$ and $Y$ be two geodesic \hm\ spaces such that $X$ is hyperbolic and satisfies (\ref{itm_Bounded1}) and~(\ref{itm_Bounded2}).
If $X$ is \qi\ to~$Y$, then $Y$ is hyperbolic.
\end{prop}

\section{Geodesic boundary}\label{sec_Boundary}

One possibility to obtain the hyperbolic boundary in case of metric spaces is to consider an equivalence relation of geodesic rays, where two geodesic rays are equivalent if they are eventually close together.
We mimic this construction for \hm\ spaces that satisfy (\ref{itm_Bounded1}) and (\ref{itm_Bounded2}) and obtain a quasiorder.
Then we use this quasiorder to define the geodesic boundary.
In the case of locally finite digraphs, we will see that this boundary is preserved by \qis, see Section~\ref{sec_qgBound}, and that it refines the ends, see Section~\ref{sec_Ends}.

Let $X$ be a hyperbolic geodesic \hm\ space.
A \emph{ray} is a map $R\colon[0,\infty)\to X$ that is continuous with respect to the forward and backward topologies and such that for every $x\in X$ and $r\geq 0$ there exists $p\geq 0$ such that $R([p,\infty))\cap B^+_r=\es$, i.\,e.\ $R$ leaves every out-ball of finite radius eventually.
An \emph{anti-ray} is a map $R\colon(-\infty,0]\to X$ that is continuous with respect to the forward and backward topologies and such that for every $x\in X$ and $r\geq 0$ there exists $p\leq 0$ such that $R((-\infty, p])\cap B^-_r=\es$.
For the sake of simplicity, we also denote by $R(i)$ the point $R(-i)$ for an anti-ray $R$ and $i> 0$.
Note that this definition of rays and anti-rays in the case of digraphs canonically corresponds to the one of Section~\ref{sec_digraphs}.

For geodesic rays or anti-rays $R_1$ and $R_2$, we write $R_1\leq R_2$ if there exists some $M\geq 0$ such that for every $r\geq 0$ and every $x\in X$ there is a directed $R_1$-$R_2$ path of length at most~$M$ outside of $\cB^+_r(x)\cup\cB^-_r(x)$.

The following lemma is straight forward to see.

\begin{lem}\label{lem_leqQOorEquiv}
Let $X$ be a geodesic \hm\ space.
\begin{enumerate}[\rm (i)]
\item\label{itm_leqQOorEquiv_1} If for all geodesic rays $R_1$ and $R_2$ in~$X$ with $R_1\leq R_2$ there exists $m\geq 0$ such that for every $r\geq 0$ and every $x\in X$ there is a directed $R_2$-$R_1$ paths of length at most~$m$ outside of $\cB^+_r(x)\cup\cB^-_r(x)$, then $\leq$ is symmetric.
\item\label{itm_leqQOorEquiv_2} If for all geodesic rays and anti-rays $R_1$ and $R_2$ in~$X$ with $R_1\leq R_2$ there exist $M\geq 0$ and a directed subpath $P$ of~$R_1$ such that $d(x,R_2)\leq M$ for all $x$ on~$R_1$ that do not lie on~$P$, then $\leq $ is transitive.\qed
\end{enumerate}
\end{lem}

\begin{lem}\label{lem_geodBound1}
Let $X$ be a $\delta$-hyperbolic geodesic \hm\ space for some $\delta\geq 0$ that satisfies (\ref{itm_Bounded1}) and (\ref{itm_Bounded2}) for the function $f\colon\R\to\R$.
Let $M\geq 0$ and let $x_1,x_2,y_1,y_2\in X$ such that
\begin{enumerate}[\rm (1)]
\item $d(x_1,y_1)\leq M$ and $d(x_2,y_2)\leq M$;
\item $d(x_1,x_2)<\infty$;
\item either $d(y_1,y_2)<\infty$ or $d(y_2,y_1)<\infty$.
\end{enumerate}
Then the following hold.
\begin{enumerate}[\rm (i)]
\item\label{itm_geodBound1_1} Every $y_1$-$y_2$ geodesic and every $y_2$-$y_1$ geodesic lie in the out-ball of radius $(2M+5\delta)+(2M+2\delta+1)f(\delta+1)$ around any $x_1$-$x_2$ geodesic.
\item\label{itm_geodBound1_1a} If $d(y_1,y_2)<\infty$, then all points $a$ on every $x_1$-$x_2$ geodesic but those with
\[
d(x_1,a)\leq((M+6\delta+2\delta f(\delta+1))f(\delta+1)+\delta)f(\delta+1)
\]
or $d(a,x_2)\leq (M+\delta)f(\delta+1)$ lie in the out-ball of radius $7\delta+2\delta f(\delta+1)$ around any $y_1$-$y_2$ geodesic.
\end{enumerate}
\end{lem}

\begin{proof}
By the assumptions, there is a the geodesic triangle with end points $x_1$, $x_2$ and $y_2$ and sides $P_2,Q$ and $R$ such that $P_2$ is an $x_2$-$y_2$ geodesic, $Q$ is an $x_1$-$y_2$ geodesic and $R$ is an $x_1$-$x_2$ geodesic.
Let $P_1$ be an $x_1$-$y_1$ geodesic, which has length at most~$M$ by our assumption, and let $S$ be a geodesic between $y_1$ and~$y_2$.

In order to prove (\ref{itm_geodBound1_1}), set $K:=(2M+5\delta)+(2M+2\delta+1)f(\delta+1)$.
If $d(x_1,y_2)\leq M+\delta$, then Proposition~\ref{prop_reverseDistanceShortDelta1} implies
\[
\ell(S)\leq Mf(\delta+1)+(M+\delta)f(\delta+1)=(2M+\delta)f(\delta+1).
\]
Hence, $S$ lies in the out-ball of radius $((2M+\delta)f(\delta+1)+M+\delta)\leq K$ around~$x_1$.

Let us now assume that $d(x_1,y_2)> M+\delta$.
Considering the geodesic triangle with sides $R$, $P_2$ and~$Q$, we obtain that $Q$ lies in $\cB^+_\delta(R)\cup\cB^-_\delta(P_2)$.
Since every point in~$\cB^-_\delta(P_2)$ has distance at most $\delta+M$ to~$y_2$, we obtain that $Q$ lies in the $(2\delta+M)$-out-ball of~$R$.

Let $v_1,\ldots, v_n$ be points on~$Q$ such that $d(x_1,v_1)=M+\delta+1$, such that $d(v_1,y_2)=(n-1)\delta+j$ for some $0\leq j<\delta$, such that $d(v_i,v_{i+1})=\delta$ for all $i<n-1$ and such that $v_n=y_2$.
Since $Q$ lies in $\cB^+_\delta(P_1)\cup\cB^-_\delta(S)$ and the length of~$P_1$ is at most~$M$, there is for every $i\leq n$ a point $w_i$ on~$S$ with $d(v_i,w_i)\leq\delta$.
We may assume that $w_n=y_2$.
For every $i\leq n$, let $A_i$ be a $v_i$-$w_i$ geodesic and, for every $i<n$, let $B_i$ be a $v_i$-$w_{i+1}$ geodesic, which exists as the composition of $v_iQv_{i+1}$ and $A_{i+1}$ is a directed $v_i$-$w_{i+1}$ path.
Note that $\ell(B_i)\leq 2\delta$.
If $w_i$ lies on~$S$ before $w_{i+1}$, i.\,e.\ the preimage of $w_i$ is smaller than that of $w_{i+1}$, then every point on $w_iSw_{i+1}$ that lies in $\cB^-_\delta(B_i)$ has distance at most $3\delta$ to~$w_{i+1}$.
By hyperbolicity, all other points on $w_iSw_{i+1}$ lie in $\cB^+_\delta(A_i)$.
In particular, $w_iSw_{i+1}$ lies in $\cB^+_{5\delta}(v_i)$.
If $w_{i+1}$ lies on~$S$ before $w_i$, then every point on $w_{i+1}Sw_i$ that lies in $\cB^-_\delta(A_i)$ has distance at most $2\delta$ to~$w_i$.
By hyperbolicity, all other points on $w_{i+1}Sw_i$ lie in $\cB^+_\delta(B_i)$.
So we also have in this case that $w_{i+1}Sw_i$ lies in $\cB^+_{5\delta}(v_i)$.
Thus, the directed subpath of~$S$ between $w_1$ and~$y_2$ lies in $\cB^+_{5\delta}(Q)$.

Let us consider the geodesic triangle with end points $x_1$, $y_1$ and~$w_1$ with $P_1$ as one side, the directed subpath of~$S$ between $y_1$ and~$w_1$ as another side and an $x_1$-$w_1$ geodesic as third side.
Since $d(x_1,w_1)\leq M+2\delta+1$, we conclude by Proposition~\ref{prop_reverseDistanceShortDelta1} that the side between $y_1$ and~$w_1$ has length at most $(2M+2\delta+1)f(\delta+1)$.
Thus, $S$ lies in the out-ball of radius
\begin{align*}
&\max\{5\delta, M+(2M+2\delta+1)f(\delta+1),\delta+ (2M+2\delta+1)f(\delta+1)\}\\
\leq&(M+3\delta)+(2M+2\delta+1)f(\delta+1)
\end{align*}
around~$Q$.
Since we already saw that $Q$ lies in the out-ball of radius $2\delta+M$ around~$R$ in this case, we obtain that $S$ lies in the out-ball of radius
\[
(2M+5\delta)+(2M+2\delta+1)f(\delta+1)\leq K
\]
around~$R$.
Together with the first case, this proves~(\ref{itm_geodBound1_1}).

Now let us assume that $d(y_1,y_2)<\infty$.
Then $Q$ lies in $\cB^+_c(P_1\cup S)$ for $c:=6\delta+2\delta f(\delta+1)$ by Lemma~\ref{lem_parallelSideCloseToAndFrom}.
By Proposition~\ref{prop_reverseDistanceShortDelta1}, all points $b$ on~$Q$ with $d(x_1,b)>(M+c)f(\delta+1)$ lie in $\cB^+_c(S)$.
By hyperbolicity, $R$ lies in $\cB^+_\delta(Q)\cup\cB^-_\delta(P_2)$ and by Proposition~\ref{prop_reverseDistanceShortDelta1} all points $a$ on~$R$ but those with $d(a,x_2)\leq (M+\delta)f(\delta+1)$ lie in $\cB^+_\delta(Q)$.
So all points $a$ on~$R$ but those with $d(a,x_2)\leq (M+\delta)f(\delta+1)$ or $d(x_1,a)\leq ((M+c)f(\delta+1)+\delta) f(\delta+1)$ lie in $\cB^+_{c+\delta}(S)$.
This shows~(\ref{itm_geodBound1_1a}).
\end{proof}

\begin{prop}\label{prop_BoundQOorEquiv}
Let $X$ be a hyperbolic geodesic \hm\ space that satisfies (\ref{itm_Bounded1}) and~(\ref{itm_Bounded2}).
Then $\leq$ is a quasiorder.
\end{prop}

\begin{proof}
Let $R_1,R_2$ be geodesic rays or anti-rays in~$X$ with $R_1\leq R_2$ and let $M\geq 0$ be such that $R_1\leq R_2$ holds for this~$M$.
Let $x_0,x_1,\ldots$ be infinitely many points on~$R_1$ such that $x_{i+1}$ lies between $x_i$ and $x_{i+2}$ with $d(x_i,x_{i+1})\geq 1$, such that there is, for every $i\in\N$, a directed path from~$x_i$ to some point $y_i$ on~$R_2$ with $d(x_i,y_i)\leq M$ and such that all of these paths are pairwise disjoint.
We may assume that $x_0$ and $y_0$ are the starting or end points of~$R_1$ and $R_2$, respectively.
Set
\[
K:=(2M+5\delta)+(2M+2\delta+1)f(\delta+1).
\]

If $R_1$ is directed away from~$x_1$, then we apply Lemma~\ref{lem_geodBound1}\,(\ref{itm_geodBound1_1}) with $(x_0,x_i,y_0,y_i)$ as $(x_1,x_2,y_1,y_2)$ for every $i\in\nat$.
If $R_1$ is directed towards $x_1$, then we apply the same lemma with $(x_i,x_1,y_i,y_1)$ as $(x_1,x_2,y_1,y_2)$ for every $i\in\nat$.

In both situations, Lemma~\ref{lem_geodBound1}\,(\ref{itm_geodBound1_1}) implies that $R_2$ lies in the out-ball of radius $K$ around~$R_1$ and hence Lemma~\ref{lem_leqQOorEquiv}\,(\ref{itm_leqQOorEquiv_2}) implies the assertion.
\end{proof}

As corollary of the proofs of Lemma~\ref{lem_geodBound1} and Proposition~\ref{prop_BoundQOorEquiv} we obtain that we may not only choose the constant $M$ to be $6\delta$ but that all of $R_2$ but some directed subpath of finite length lies within the out-ball of radius $M$ around $R_1$.

\begin{cor}\label{cor_geoBoundAlwaysClose}
Let $X$ be a hyperbolic geodesic semimetric space that satisfies (\ref{itm_Bounded1}) and (\ref{itm_Bounded2}).
If $R_1$ and $R_2$ are geodesic (anti-)rays with $R_1\leq R_2$, then there is a (anti-)subray $R_2'$ of~$R_2$ such that $R_2'\sub \cB^+_{6\delta}(R_1)$.\qed
\end{cor}

Let $X$ be a hyperbolic geodesic \hm\ space that satisfies (\ref{itm_Bounded1}) and~(\ref{itm_Bounded2}).
If $\leq$ is a quasiorder on the set of geodesic rays and anti-rays of~$X$, then we write $R_1\approx R_2$ if $R_1\leq R_2$ and $R_2\leq R_1$.
This new relation is an equivalence relation whose equivalence classes form the \emph{geodesic boundary} $\rand_{geo}X$ of~$X$.
We define two related boundaries: the \emph{geodesic f-boundary} $\rand_{geo}^fX$ consists of the equivalence classes of $\approx$ restricted to the rays and the \emph{geodesic b-boundary} $\rand_{geo}^bX$ consists of the equivalence classes of $\approx$ restricted to the anti-rays.
Every geodesic boundary point is the union of at most one geodesic f-boundary point and at most one geodesic b-boundary point.

Note that $\leq$ extends to an order on the three sets $\rand_{geo} X$, $\rand_{geo}^f$ and $\rand_{geo}^b$.
Lemma \ref{lem_geodBound1} enables us to prove some order-theoretic results on the geodesic boundary.

\begin{prop}\label{prop_geoBoundNo3Chain}
Let $X$ be a hyperbolic geodesic \hm\ space that satisfies (\ref{itm_Bounded1}) and~(\ref{itm_Bounded2}).
Let $\eta,\mu\in\rand _{geo}X$ with $\eta<\mu$.
Then either $\eta$ or~$\mu$ contains no ray and the other one contains no anti-ray.

In particular, there are no chains of length at least~$3$ in $\rand_{geo}X$.
\end{prop}

\begin{proof}
Let $\eta,\mu\in\rand_{geo}X$ with $\eta\leq\mu$ and let $R\in\eta$ and $Q\in\mu$.
Assume that either both, $R$ and~$Q$, are rays or that both are anti-rays.
Then we use the method of the proof of Proposition~\ref{prop_BoundQOorEquiv} and apply Lemma~\ref{lem_geodBound1}\,(\ref{itm_geodBound1_1a}) to conclude $Q\leq R$.
So we have $Q\approx R$ and hence $\eta=\mu$.
\end{proof}

It would be interesting to know whether divergence of geodesics or geodesic stability is strong enough to give rise to a geodesic boundary.

\section{Quasi-geodesic boundary}\label{sec_qgBound}

In this section, we define a different boundary that builds upon a quasiorder on the set of \qg\ rays and anti-rays.
For hyperbolic digraphs satisfying (\ref{itm_Bounded1}) and (\ref{itm_Bounded2}) this new boundary will coincide with the geodesic boundary, see Proposition~\ref{prop_qgBPContainsGeodesic}.
As a corollary we obtain that \qis\ preserve the structure of the geodesic boundary, see Theorem~\ref{thm_QIPreserveBoundaryDelta1}.

We extend $\leq$ to the class of \qg\ rays and anti-rays:
for \qg\ rays or anti-rays $R_1$ and~$R_2$ in a hyperbolic geodesic \hm\ space $X$, we write $R_1\leq R_2$ if there exists some $M\geq 0$ such that for every $r\geq 0$ and every $x\in X$ there is a directed $R_1$-$R_2$ path of length at most~$M$ outside of $\cB^+_r(x)\cup\cB^-_r(x)$.

\begin{prop}\label{prop_QGBoundDelta1}
For every hyperbolic geodesic \hm\ space $X$ that satisfies (\ref{itm_Bounded1}) and (\ref{itm_Bounded2}), the relation $\leq$ is a quasiorder on the set of \qg\ rays and anti-rays.
\end{prop}

\begin{proof}
Let $f\colon\R\to\R$ be a function such that $D$ satisfies (\ref{itm_Bounded1}) and (\ref{itm_Bounded2}) for~$f$.
Let $\gamma\geq 1$ and $c\geq 0$.
Let $R_1$ and $R_2$ be $(\gamma,c)$-\qg\ rays or anti-rays in~$D$ such that $R_1\leq R_2$.
Let $x$ be the starting or end point of~$R_1$ and $y$ be the starting or end point of~$R_2$.
Let $M\geq 0$ and let $x_0,x_1,\ldots$ be points on~$R_1$ and $y_0,y_1,\ldots$ be points on~$R_2$ such that $d(x_i,y_i)\leq M$ and $d^\lrarrow(x,x_i)\geq i$ and $d^\lrarrow(y,y_i)\geq i$ for all $i\geq 0$.
We may assume $x=x_0$ and $y=y_0$.
Let $\kappa\geq 0$ such that geodesic stability holds for $(\gamma,c)$-\qg s with respect to the value~$\kappa$, cp.\ Theorem~\ref{thm_Delta1GeodStab}.

Let $R_1^i$ be the subpath of~$R_1$ between $x$ and~$x_i$ and let $R_2^i$ be the subpath of~$R_2$ between $y$ and~$y_i$.
Let $P_i$ be a geodesic with the same starting point as $R_1^i$ and the same end point as~$R_1^i$ and let $Q_i$ be a geodesic with the same starting point as $R_2^i$ and the same end point as~$R_2^i$.
We apply Lemma~\ref{lem_geodBound1}\,(\ref{itm_geodBound1_1}) for the four points $x,x_i,y,y_i$ for every $i\geq 1$ and use geodesic stability to conclude that $R_2$ lies in the ball of radius
\[
2\kappa+(2M+5\delta)+(2M+2\delta+1)f(\delta+1)
\]
around~$R_1$.
This shows with an analogue to Lemma~\ref{lem_leqQOorEquiv}\,(\ref{itm_leqQOorEquiv_2}) for the relation $\leq$ on \qg\ rays and anti-rays instead of geodesic ones that $\leq$ is transitive.
Since the relation is obviously reflexive, the assertion follows.
\end{proof}

Similar to Corollary~\ref{cor_geoBoundAlwaysClose}, we may choose $M$ such that some (anti-)subray $R_2'$ of $R_2$ lies within $\cB^+_M(R_1')$.
But contrary to that corollary, in the case of the \qg\ boundary, this constant also depends on the \qg\ constants of $R_1$ and~$R_2$.

Similar to the case of geodesic rays and anti-rays, if $\leq$ is a quasiorder on the set of \qg\ rays and anti-rays, we write $R_1\approx R_2$ if $R_1\leq R_2$ and $R_2\leq R_1$, where $R_1$ and $R_2$ are \qg\ rays or anti-rays in a hyperbolic geodesic \hm\ space $X$ that satisfies (\ref{itm_Bounded1}) and~(\ref{itm_Bounded2}).
Then $\approx$ is an equivalence relation whose equivalence classes form the \emph{\qg\ boundary} $\rand X$ of the \hm\ space.
The \emph{\qg\ f-boundary}  $\rand^f X$ are the equivalence classes of $\approx$ restricted to the \qg\ rays and the \emph{\qg\ b-boundary}  $\rand^b X$ are the equivalence classes of $\approx$ restricted to the \qg\ anti-rays.
We also have in this case that every \qg\ boundary point is the union of at most one \qg\ f-boundary point with at most one \qg\ b-boundary point.
Note that we can extend the quasiorder $\leq$ to an order on~$\rand X$, on~$\rand^f X$ and on~$\rand^b X$.
Analogous to Proposition~\ref{prop_geoBoundNo3Chain}, we obtain the following result.

\begin{prop}\label{prop_BoundNo3Chain}
Let $X$ be a hyperbolic geodesic \hm\ space that satisfies (\ref{itm_Bounded1}) and~(\ref{itm_Bounded2}).
Let $\eta,\mu\in\rand X$ with $\eta<\mu$.
Then either $\eta$ or~$\mu$ contains no ray and the other one contains no anti-ray.

In particular, there are no chains of length at least~$3$ in $\rand X$.\qed
\end{prop}

In hyperbolic geodesic spaces, we can apply the Arzel\`a-Ascoli theorem to prove that every \qg\ ray lies close to and from some geodesic ray, see e.\,g.\ \cite[Lemma I.8.28]{BridsonHaefliger}.
Generally, in the case of \hm\ spaces an analogue of the Arzel\`a-Ascoli theorem is false, but gets true with additional strong requirements, see~\cite{CZ-Arzela-Ascoli}.
Since we do not satisfy these additional requirements in general, we prove the desired result on \qg\ and geodesic rays and anti-rays only in the case of digraphs, where we can apply elementary arguments instead of the Arzel\`a-Ascoli theorem.

\begin{prop}\label{prop_qgBPContainsGeodesic}
Let $D$ be a hyperbolic digraph satisfying (\ref{itm_Bounded1}) and~(\ref{itm_Bounded2}).
Then the following hold.
\begin{enumerate}[\rm (i)]
\item\label{itm_qgBPContainsGeodesic1} If all vertices of~$D$ have finite out-degree, then every \qg\ f-boundary point contains a geodesic f-boundary point.
\item\label{itm_qgBPContainsGeodesic2} If all vertices of~$D$ have finite in-degree, then every \qg\ b-boundary point contains a geodesic b-boundary point.
\item\label{itm_qgBPContainsGeodesic3} If $D$ is locally finite, then every \qg\ boundary point contains a geodesic boundary point.
\end{enumerate}
\end{prop}

\begin{proof}
Let us assume that every vertex of~$D$ has finite out-degree and let $Q=x_0x_1\ldots$ be a $(\gamma,c)$-\qg\ ray in~$D$ for some $\gamma\geq 1$ and $c\geq 0$.
For every $i\in\N$, let $P_i$ be an $x_0$-$x_i$ geodesic.
Since $x_0$ has finite out-degree, infinitely many $P_i$ have a common first edge.
Similarly, among those there are again infinitely many with a common second edge and so on.
This way we obtain a ray $R$ with starting vertex~$x_0$ and such that every finite directed subpath is geodesic.
Hence, $R$ is geodesic as well.

By geodesic stability, see Theorem~\ref{thm_Delta1GeodStab}, there is some $\kappa\geq 0$ such that every $P_i$ lies in the out-ball and in-ball of radius~$\kappa$ around $x_0Qx_i$.
Thus, $R$ lies in the out-ball and in-ball of radius~$\kappa$ around~$Q$.
Thus, we have $Q\leq R$ and $R\leq Q$.
This shows (\ref{itm_qgBPContainsGeodesic1}).

By a symmetric argument with all directions of the edges reversed, we obtain (\ref{itm_qgBPContainsGeodesic2}) and (\ref{itm_qgBPContainsGeodesic3}) follows immediately from (\ref{itm_qgBPContainsGeodesic1}) and (\ref{itm_qgBPContainsGeodesic2}).
\end{proof}

The advantage of the \qg\ boundary is that it is preserved by \qis\ since \qis\ preserve the set of \qg\ rays and anti-rays.
In the case of locally finite digraphs, Proposition~\ref{prop_qgBPContainsGeodesic} implies that \qis\ also preserve the geodesic boundary.
Thus, we immediately have the following results.

\begin{thm}\label{thm_QIPreserveBoundaryDelta1}
Let $f\colon X_1\to X_2$ be a \qiy\ between hyperbolic geodesic \hm\ spaces $X_1$ and~$X_2$ that satisfy (\ref{itm_Bounded1}) and~(\ref{itm_Bounded2}).
Then $f$ canonically defines three order-preserving bijective maps: one between the quasi-geodesic f-boundaries, one between the quasi-geodesic b-boundaries and one between the quasi-geodesic boundaries.\qed
\end{thm}

\begin{thm}\label{thm_QIPreserveGeodBoundaryDigraphs}
Let $f\colon D_1\to D_2$ be a \qiy\ between hyperbolic digraphs $D_1$ and~$D_2$ that satisfy (\ref{itm_Bounded1}) and~(\ref{itm_Bounded2}).
Then the following hold.
\begin{enumerate}[\rm(i)]
\item If every vertex of~$D_1$ and~$D_2$ has finite out-degree, then $f$ canonically defines an order-preserving bijective map between the geodesic f-boundaries of the digraphs.
\item If every vertex of~$D_1$ and~$D_2$ has finite in-degree, then $f$ canonically defines an order-preserving bijective map between the geodesic b-boundaries of the digraphs.
\item If $D_1$ and~$D_2$ are locally finite, then $f$ canonically defines an order-preserving bijective map between the geodesic boundaries of the digraphs.\qed
\end{enumerate}
\end{thm}

\section{Ends of digraphs}\label{sec_Ends}

In this section, we will briefly introduce the notion of ends of digraphs as introduced by Zuther~\cite{Z-EndsDigraphs} and then show that the geodesic boundary of hyperbolic locally finite digraphs is a refinement of the set of ends.
We note that B\"urger and Melcher~\cite{BM-EndsOfDigraphsI,BM-EndsOfDigraphsII,BM-EndsOfDigraphsIII} recently investigated a different notion of ends of digraphs.
Roughly speaking, their ends are those ends in Zuther's sense that contain rays and anti-rays.
In general, the geodesic boundary is not a refinement of the ends in sense of B\"urger and Melcher: while each of their ends still contains a geodesic boundary point, there may be geodesic boundary points no belonging to any of their ends.
Jackson and Kilibarda~\cite{JK-EndsForSemigroupsAndMonoids} used a different notion for ends of semigroups that is based on the ends of the underlying undirected graph of their Cayley digraphs.
Gray and Kambites~\cite{GK-SemimetricSpaces} proved that the ends in the sense of Jackson and Kilibarda are invariant under \qis.
In this section, we will also show that Zuther's notion of ends of digraphs is preserved by \qis\ in the case of locally finite digraphs.

Our main interest in this section is to prove that the geodesic boundary of hyperbolic locally finite digraphs is a refinement of their ends.
By reasons addressed in the previous section, we do not obtain this result for \hm\ spaces: apart from a notion of ends for \hm\ spaces, we would need a suitable notion of the Arzelà-Ascoli theorem.

In order to define the ends of digraphs, we first define a relation on the set $\cR$ of all rays and anti-rays in a digraph~$D$.
For $R_1,R_2\in\cR$, we write $R_1 \preccurlyeq R_2$ if there are infinitely many pairwise disjoint $R_1$-$R_2$ paths in~$D$.
Zuther~\cite[Proposition 2.2]{Z-EndsDigraphs} showed that $\preccurlyeq$ is a quasiorder on~$\cR$.
We write $R_1\sim R_2$ if $R_1 \preccurlyeq R_2$ and $R_2 \preccurlyeq R_1$.
This is an equivalence relation whose classes are the \emph{ends} of~$D$.
We denote this set by $\Omega D$.
Its restrictions to either the rays or the anti-rays are also equivalence relations whose classes are the \emph{f-ends} or \emph{b-ends}, respectively.
Every end is the union of at most one f-end and at most one b-end.
Note that $\preccurlyeq$ extends to an order on the set of ends, of f-ends and of b-ends of~$D$.

In the case of a graph $G$, it is easy to see that two rays have infinitely many pairwise disjoint paths between them if and only if for every finite vertex set $S$ those rays have tails that lie in the same component of $G-S$.
Our next result shows proves an analogous result for digraphs.

\begin{prop}\label{prop_equivDefEnds}
Let $D$ be a digraph.
\begin{enumerate}[\rm (i)]
\item\label{itm_equivDefEnds1} Let every vertex of~$D$ have finite out-degree and let $R_1,R_2$ be two rays in~$D$.
Then $R_1\preccurlyeq R_2$ if and only if for every $x\in V(D)$ and every $R\in\N$ there is a directed $R_1$-$R_2$ path outside of $\cB^+_R(x)$.
\item\label{itm_equivDefEnds2} Let every vertex of~$D$ have finite in-degree and let $R_1,R_2$ be two anti-rays in~$D$.
Then $R_1\preccurlyeq R_2$ if and only if for every $x\in V(D)$ and every $R\in\N$ there is a directed $R_1$-$R_2$ path outside of $\cB^-_R(x)$.
\item\label{itm_equivDefEnds3} Let $D$ be locally finite and let $R_1$ and $R_2$ be rays or anti-rays in~$D$.
Then $R_1\preccurlyeq R_2$ if and only if for every $x\in V(D)$ and every $R\in\N$ there is a directed $R_1$-$R_2$ path outside of $\cB^+_R(x)\cup\cB^-_R(x)$.
\end{enumerate}
\end{prop}

\begin{proof}
To prove (\ref{itm_equivDefEnds1}), let $x\in V(D)$ and $r\in\N$.
If $R_1\preccurlyeq R_2$, then we have infinitely many pairwise disjoint directed $R_1$-$R_2$ paths.
All but finitely many of them do not meet $\cB^+_r(x)$, since that is a finite set.

To prove the remaining direction of~(\ref{itm_equivDefEnds1}), let $x$ be the starting vertex of~$R_1$.
Suppose that there are only finitely many pairwise disjoint directed $R_1$-$R_2$ paths.
Since for each of those finitely many directed paths $P$ there is a directed subpath of~$R_1$ from~$x$ to the starting vertex of~$P$, there exists $r\in\N$ such that all of those directed $R_1$-$R_2$ paths lie in $\cB^+_r(x)$.
This is a contradiction to the assumption that there is a directed $R_1$-$R_2$ path outside of $\cB^+_r(x)$.

With a similar argument but with $x$ being the last vertex of the anti-ray~$R_2$ for the reverse direction, we obtain~(\ref{itm_equivDefEnds2}).

Let us prove (\ref{itm_equivDefEnds3}).
If there are infinitely many pairwise disjoint directed $R_1$-$R_2$ paths, then for every $x\in V(D)$ and every $r\in\N$, there is an $R_1$-$R_2$ path outside of the finite set $\cB^+_r(x)\cup\cB^-_r(x)$.
To prove the other direction, it suffices to consider the case that $R_1$ is an anti-ray and $R_2$ a ray, since we can follow the proof of (\ref{itm_equivDefEnds1}) if~$R_1$ is a ray and the proof of~(\ref{itm_equivDefEnds2}) if $R_2$ is an anti-ray.
So let us assume that for every $x\in V(D)$ and every $r\in\N$ there is a directed $R_1$-$R_2$ path outside of $\cB^+_r(x)\cup\cB^-_r(x)$.
Let us suppose that there are only finitely many pairwise disjoint directed $R_1$-$R_2$ paths $P_1,\ldots,P_n$.
Let $x_i$ be the starting vertex of~$P_i$ for all $1\leq i\leq n$ and let $a$ be the end vertex of~$R_1$.
Set $N:=\max\{\ell(P_i)\mid 1\leq i\leq n\}$ and $r:=N+\max\{d(x_i,a)\mid 1\leq i\leq n\}$.
Let $x\in\{x_1,\ldots,x_n\}$ with $d(x,a)=r-N$.
Then $P_1,\ldots,P_n$ lie in $\cB^+_r(x)$.
By assumption, we find an $R_1$-$R_2$ path $P$ outside of $\cB^+_r(x)\cup\cB^-_r(x)$.
This is disjoint to all paths $P_i$ by its choice.
This contradiction shows $R_1\preccurlyeq R_2$.
\end{proof}

Craik et al.\ \cite[Corollary 2.3]{CGKMR-EndsSemigroups} proved that Zuther's definition of ends of digraphs extends to a notion of ends of finitely generated semigroups that is preserved under changing the finite generating set.
More precisely, for a finitely generated semigroup every right Cayley digraph has the same isomorphism type (as a partially ordered set) of their ends.
We prove a similar result for \qis\ of digraphs.

\begin{thm}\label{thm_endsQIInvariant}
\mbox{}
\begin{enumerate}[\rm (i)]
\item\label{itm_endsQIInvariant1} Quasi-isometries between digraphs all of whose vertices have finite out-degree extend canonically to bijective maps between their f-ends that preserve the order~$\preccurlyeq$.
\item\label{itm_endsQIInvariant2} Quasi-isometries between digraphs all of whose vertices have finite in-degree extend canonically to bijective maps between their b-ends that preserve the order~$\preccurlyeq$.
\item\label{itm_endsQIInvariant3} Quasi-isometries between locally finite digraphs extend canonically to bijective maps between their ends that preserve the order~$\preccurlyeq$.
\end{enumerate}
\end{thm}

\begin{proof}
Let $D_1$ and~$D_2$ be digraphs and let $f\colon D_1\to D_2$ be a $(\gamma,c)$-\qiy\ for some $\gamma\geq 1$ and $c\geq 0$.
Let $R_1=x_1x_2\ldots$ be a ray in~$D_1$.
Then there is a directed $f(x_i)$-$f(x_{i+1})$ path $P^1_i$ of length at most $\gamma+c$ in~$D_2$.
Concatenating all these directed paths leads to a one-way infinite directed path~$W_1$.
This contains a ray~$Q_1$.
Note that there are infinitely many pairwise disjoint $Q_1$-$\{f(x_i)\mid i\in\N\}$ paths and infinitely many pairwise disjoint $\{f(x_i)\mid i\in\N\}$-$Q_1$ paths.
So all possible rays obtained the same way as~$Q_1$ lie in the same f-end.
In the same way, we obtain for a second ray $R_2=y_1y_2\ldots$ in~$D_1$ paths $P_i^2$ from~$f(y_i)$ to~$f(y_{i+1})$ of length at most $\gamma+c$ and a one-way infinite directed path~$W_2$ and a ray~$Q_2$ in~$W_2$.

Let us assume that $R_1\preccurlyeq R_2$ and that all vertices of~$D_1$ and~$D_2$ have finite out-degree.
Let $x\in V(D_2)$ and $n\in\N$.
Then $W_1$, $W_2$, $Q_1$ and~$Q_2$ have all but finitely many of their vertices outside of $\cB^+_n(x)$.
Let $Q_1'$ and $Q_2'$ be tails of~$Q_1$ and $Q_2$, respectively, such that $Q_1'$ and $Q_2'$ do not meet $\cB^+_n(x)$.
Let $N\in\N$ such that for all $i\geq N$ there is a $Q_1'$-$f(x_i)$ path and an $f(y_i)$-$Q_2'$ path outside of $\cB^+_n(x)$.
Let $y\in V(D_1)$ such that $d(f(y),x)\leq c$ and $d(x,f(y))\leq c$.
There exists a directed $R_1$-$R_2$ path $P$ outside of $\cB^+_{(\gamma+c)(n+1)+c}(y)$ by Proposition~\ref{prop_equivDefEnds}.
We may assume that it is from $x_i$ to $y_j$ for some $i,j\geq N$.
Its $f$-image induces a directed $f(x_i)$-$f(y_j)$ path outside of $\cB^+_n(x)$.
By the choice of~$i$ and~$j$, we find a $Q_1'$-$Q_2'$ path outside of $\cB^+_n(x)$.
This shows $Q_1\preccurlyeq Q_2$.

Note that there is a \qiy\ $g\colon D_2\to D_1$ such that $g\circ f$ is almost the identity: that there exists some $\ell\geq 0$ with $d(u,g(f(u)))\leq \ell$ for all $u\in V(D_1)$ and $d(v,f(g(v)))\leq\ell$ for all $v\in V(D_2)$.
This implies that $R_1\preccurlyeq R_2$ if and only if $Q_1\preccurlyeq Q_2$ and finishes the proof of~(\ref{itm_endsQIInvariant1}).

To prove (\ref{itm_endsQIInvariant2}), we essentially follow the proof of~(\ref{itm_endsQIInvariant1}) with reversed directions of the edges.
So e.\,g.\ the paths $P_i^1$ go from $f(x_{i+1})$ to~$f(x_i)$ and the distances are measured towards~$x$ instead of from~$x$.
Also, the proof of~(\ref{itm_endsQIInvariant3}) is essentially the same.
\end{proof}

Let us now prove that the geodesic boundary is a refinement of the ends.

\begin{prop}\label{prop_boundaryIsRefinement}
For every hyperbolic digraph $D$ that satifies (\ref{itm_Bounded1}) and (\ref{itm_Bounded2}) the following hold.
\begin{enumerate}[\rm (i)]
\item\label{itm_boundaryIsRefinement_1} If every vertex of~$D$ has finite out-degree, then the geodesic f-boundary is a refinement of the f-ends.
\item\label{itm_boundaryIsRefinement_2} If every vertex of~$D$ has finite in-degree, then the geodesic b-boundary is a refinement of the b-ends.
\item\label{itm_boundaryIsRefinement_3} If $D$ is locally finite, then the geodesic boundary is a refinement of the ends.
\end{enumerate}
\end{prop}

\begin{proof}
Let us assume that every vertex has finite out-degree and let $f\colon\R\to\R$ such that (\ref{itm_Bounded1}) and (\ref{itm_Bounded2}) are satisfied for the function~$f$.

Let $R=x_0x_1\ldots$ be a ray in~$D$.
For every $i\in\N$, let $P_i$ be an $x_0$-$x_i$ geodesic.
These paths give rise to a ray in that infinitely many directed paths $P_i$ have a common first edge since $x_0$ has finite out-degree, among which there are again infinitely many $P_i$ that also share a common second edge and so on.
The obtained ray $Q=y_0y_1\ldots$ with $x_0=y_0$ is geodesic since each finite directed subpath is contained in some~$P_i$.
Obviously, we have $Q\preccurlyeq R$.

Let $k\in\N$ and set $i_k:=d(x_0,x_k)$.
Let $\ell\in\N$ such that $y_0Qy_{i_k+2\delta+1}$ is a subpath of~$P_\ell$.
Let us consider the geodesic triangle with sides $P_k$, $P_\ell$ and an $x_k$-$x_\ell$ geodesic~$P$.
Let $v$ be the last vertex of~$P$ in $\cB^+_\delta(P_k)$ and let $w$ be its out-neighbour on~$P$.
By $\delta$-hyperbolicity, we know that $w$ lies in $\cB^-_\delta(u_k)$ for some $u_k$ on~$P_\ell$.
It follows that $d(x_0,u_k)\leq i_k+2\delta+1$, so $u_k$ lies on~$Q$.
Thus, the concatenation of $x_kPw$ with a $w$-$u_k$ geodesic shows the existence of an $x_k$-$u_k$ geodesic~$Q_k$.
So we find infinitely many directed $R$-$Q$ paths.
It remains to show that we can find infinitely many pairwise disjoint ones.

Let $\cP$ be a maximal set of these directed paths~$Q_i$ that are pairwise disjoint.
Suppose that $\cP$ is finite.
Let $i$ be the maximum distance from~$x_0$ to vertices on elements of~$\cP$ and set $n:=(\delta+i)f(\delta+1)+1$.
Let $k\in\N$ such that $y_0Qy_n$ is a subpath of~$P_k$.
Considering the geodesic triangle with sides $y_nP_kx_k$, $Q_k$ and $y_nQu_k$, we get by $\delta$-hyperbolicity that $Q_k$ lies in $\cB^+_\delta(y_nP_kx_k)\cup\cB^-_\delta(y_nQu_k)$.
Let us suppose that $Q_k$ had a vertex $v$ in~$\cB^+_i(x_0)$.
If $v\in \cB^+_\delta(y_nP_kx_k)$, then for $w\in V(y_nP_kx_k)$ with $d(w,v)\leq\delta$, Proposition~\ref{prop_reverseDistanceShortDelta1} implies $d(x_0,w)\leq (\delta+i)f(\delta+1)$, which contradicts that $P_k$ is a geodesic and thus $d(x_0,w)\geq (\delta+i)f(\delta+1)+1$.
So $v$ lies in $\cB^-_\delta(u)$ for some $u\in V(y_nQu_k)$.
But then $d(x_0,u)\leq i+\delta$, which is also a contradiction.
Thus, $Q_k$ is disjoint to every element of~$\cP$, which contradicts the maximality of~$\cP$.
Thus, there are infinitely many pairwise disjoint $R$-$Q$ paths and hence we have $R\preccurlyeq Q$.
So $R$ and~$Q$ lie in the same f-end of~$D$.
This shows~(\ref{itm_boundaryIsRefinement_1}).

Analogously with the directions of the edges reversed, we obtain~(\ref{itm_boundaryIsRefinement_2}).
Finally, (\ref{itm_boundaryIsRefinement_3}) is an immediate consequence of~(\ref{itm_boundaryIsRefinement_1}) and~(\ref{itm_boundaryIsRefinement_2}).
\end{proof}

We will use the map of Proposition~\ref{prop_boundaryIsRefinement}\,(\ref{itm_boundaryIsRefinement_3}) in Section~\ref{sec_size} to understand the connection between the ends and the (quasi-)geodesic boundary in more detail.

\section{Topologies for the \qg\ boundary}\label{sec_topology}

In this section, we extend the two topologies $\cO_f$ and $\cO_b$ to the boundary of hyperbolic geodesic \hm\ spaces and show that \qis\ extend to homeomorphisms with respect to both topologies on the boundaries.
Let $X$ be a hyperbolic geodesic \hm\ space that satisfies (\ref{itm_Bounded1}) and (\ref{itm_Bounded2}).
Let $x\in X$, let $r\geq 0$ and let $\omega\in\rand^f X$.
We set
\[
C^-(\omega,x,r):=\{y\in X\mid\exists R\in\omega\ \forall z\text{ on } R\ \exists y\text{-}z\text{ geodesic outside of }\cB^+_r(x)\cup\cB^-_r(x)\}.
\]
Analogously, if $\eta\in\rand^b X$, then we set
\[
C^+(\eta,x,r):=\{y\in X\mid\exists R\in\omega\ \forall z\text{ on } R\ \exists z\text{-}y\text{ geodesic outside of }\cB^+_r(x)\cup\cB^-_r(x)\}.
\]

For $\mu\in\rand X$ with $\mu=\mu_1\cup\mu_2$ for $\mu_1\in\rand^f X$ and $\mu_2\in\rand^b X$, we set
\[
C^-(\mu,x,r):=C^-(\mu_1,x,r)
\]
and
\[
C^+(\mu,x,r):=C^+(\mu_2,x,r).
\]
We say that an element $\mu'$ of $\rand X\cup\rand^f X\cup\rand^b X$ \emph{lives in} $C^-(\mu,x,r)$ or $C^+(\mu,x,r)$ if an element of~$\mu'$ lies in $C^-(\mu,x,r)$ or $C^+(\mu,x,r)$, respectively.
We denote by $C^-_\rand(\mu,x,r)$ and $C^+_\rand(\mu,x,r)$ the sets $C^-(\mu,x,r)$ and $C^+(\mu,x,r)$ together with the \qg\ boundary points living in them.

Generally, not all \qg\ boundary points that live in $C^-(\omega,x,r)$ or $C^+(\omega,x,r)$ have the property that each of their elements contains a subray or anti-subray that is contained in $C^-(\omega,x,r)$ or $C^+(\omega,x,r)$, respectively.
However, we shall show that this is true up to small changes on the constant~$r$.

\begin{lem}\label{lem_neighbourhoodsContainWholeBP}
Let $X$ be a hyperbolic geodesic \hm\ spaces that satisfies (\ref{itm_Bounded1}) and (\ref{itm_Bounded2}).
Let $x\in X$, let $r\geq 0$ and let $\omega\in\rand^f X\cup\rand^b X$.
Set $\kappa:=6\delta+2\delta f(\delta+1)$.
Then the following hold.
\begin{enumerate}[\rm (i)]
\item\label{itm_neighbourhoodsContainWholeBP_1} If $\omega\in \rand^f X$ and $R$ is a \qg\ ray or anti-ray in $C^-(\omega,x,r)$, then every \qg\ ray or anti-ray $Q\approx R$ lies in $C^-(\omega,x,r-\kappa)$ \emph{eventually}, i.\,e.\ there is at most a directed subpath of~$Q$ of finite length outside of $C^-(\omega,x,r-\kappa)$.
\item\label{itm_neighbourhoodsContainWholeBP_2} If $\omega\in \rand^b X$ and $R$ is a \qg\ ray or anti-ray in $C^+(\omega,x,r)$, then every \qg\ ray or anti-ray $Q\approx R$ lies in $C^+(\omega,x,r-\kappa)$ eventually.
\end{enumerate}
\end{lem}

\begin{proof}
By symmetry, it suffices to prove (\ref{itm_neighbourhoodsContainWholeBP_1}).
By (\ref{itm_Bounded1}), we know that at most some directed subpath of finite length of~$Q$ has its starting and end point in $\cB^+_r(x)\cup\cB^-_r(x)$.
So we may assume that $Q$ lies outside of $\cB^+_r(x)\cup \cB^-_r(x)$.
Let $M\geq 0$ such that $Q\preccurlyeq R$ holds with respect to the constant~$M$.
If a $Q$-$R$ geodesic $P$ of length at most~$M$ contains a point of $\cB^+_r(x)\cup\cB^-_r(x)$, then either its end point lies in $\cB^+_{r+M}(x)$ or its starting point lies in $\cB^-_{r+M}(x)$.
Since there is at most a directed subpath of~$R$ of finite length with its starting and end point in $\cB^-_{r+M}(x)$ by~(\ref{itm_Bounded1}), we may replace $R$ by a subray or anti-subray that lies outside of $\cB^+_{r+M}(x)\cup\cB^-_{r+M}(x)$ and hence the first case does not happen.
Analogously, we may replace $Q$ by a subray or anti-subray of~$Q$ that avoids $\cB^+_{r+M}(x)\cup\cB^-_{r+M}(x)$ by~(\ref{itm_Bounded2}).
Thus, no $Q$-$R$ geodesic intersects $\cB^+_r(x)\cup\cB^-_r(x)$.
Applying Lemma~\ref{lem_parallelSideCloseToAndFrom} shows that if two composable geodesics lie outside of $\cB^+_r(x)\cup \cB^-_r(x)$, then any geodesic that is parallel to the composition lies outside of $\cB^+_{r-\kappa}(x)\cup \cB^-_{r-\kappa}(x)$.
This shows (\ref{itm_neighbourhoodsContainWholeBP_1}).
\end{proof}

Now we are able to define a base for the topologies on $X\cup\rand X$.
Let $\omega\in\rand X$.
We set
\[
C^f_\rand(\omega,x,r):=\bigcup_{\mu\in\rand X,\omega\leq \mu}C^+_\rand(\mu,x,r)
\]
for all $x\in X$ and $r\geq 0$.
Then we declare all sets $C^f_\rand(\omega,x,r)$ as open.
These sets together with the open balls $\mathring{\cB}^+_r(x)$ form a base for the topology $\cO_f$ of $X\cup\rand X$.
Analogously, we set
\[
C^b_\rand(\omega,x,r):=\bigcup_{\mu\in\rand X,\mu\leq\omega}C^-_\rand(\mu,x,r)
\]
for all $x\in X$ and $r\geq 0$ and obtain a base for the topology $\cO_b$ of $X\cup\rand X$ if we declare the sets $C^b_\rand(\omega,x,r)$ as open and take them together with the open balls~$\mathring{\cB}^-_r(x)$.

We denote by $\rand^+(\omega)$, by $\rand^-(\omega)$, the elements $\eta$ of~$\rand X$ with $\omega\leq\eta$, with $\eta\leq\omega$, respectively.
It immediately follows from the above definition that every f-neighbourhood of~$\omega$ contains~$\rand^+(\omega)$ and every b-neighbourhood of~$\omega$ contains~$\rand^-(\omega)$.

Let us illustrate the definition of the topologies via the following example.

\begin{ex}\label{ex_topBound}
Let $D$ be the digraph with distinct vertices $u_i, v_i, w_i, x_i, y_i$ for all $i\in\N$.
The edges are the following:
\begin{enumerate}[--]
\item for every $i\in\N$ we have the edges $u_iv_i$, $v_iw_i$, $w_ix_i$ and~$x_i,y_i$,
\item for every $i\in\N$ there is an edge $v_iv_{i+1}$ and
\item for every $i\in\N$ there is an edge $x_{i+1}x_i$.
\end{enumerate}
Then $D$ is hyperbolic with one \qg\ boundary point $\eta$ being the equivalence class of the ray $v_0v_1\ldots$ and the only other \qg\ boundary point $\mu$ being the equivalence class of the anti-ray $\ldots x_1x_0$.
Then the typical open b-neighbourhood of~$\eta$ that lies in the defined base consists of~$\eta$ and the vertices $u_i$ and $v_i$ for all $i\geq i_0$ for some $i_0\in\N$.
The (typical) open f-neighbourhoods of~$\eta$ in the base consist of~$\eta$ and~$\mu$ and the vertices $x_i$ and~$y_i$ for all $i\geq i_0$ for some $i_0\in\N$.
The neighbourhoods of~$\mu$ are obtained symmetrically with $u_i$ swapped with $y_i$ and $v_i$ swapped with~$x_i$.
At first, it may be surprising that none of the vertices $w_i$ lie in these typical neighbourhoods, but intuitively, there are no directed $w_i$-$\eta$ or $\mu$-$w_i$ paths.
\end{ex}

\begin{thm}\label{thm_QisImplyHomeo}
Let $f\colon X_1\to X_2$ be a \qiy\ between hyperbolic geodesic \hm\ spaces that satisfy (\ref{itm_Bounded1}) and~(\ref{itm_Bounded2}).
Then $f$ canonically defines a map $\widehat f\colon \rand X_1\to \rand X_2$ that is a homeomorphism with respect to $\cO_f$ and $\cO_b$.
\end{thm}

\begin{proof}
Let $\gamma\geq 1$ and $c\geq 0$ such that $f$ is $(\gamma,c)$-\qg.
Theorem~\ref{thm_QIPreserveBoundaryDelta1} implies that $f$ canonically defines an order-preserving bijective map $\widehat f\colon \rand X_1\to \rand X_2$.
Let us consider a non-trivial set $C^+(\eta,x,r)$ with $\eta\in\rand X$, $x\in X$ and $r\geq 0$.
Then
\[
f(C^+(\eta,x,r))\sub C^+(\widehat f(\eta), f(x), \frac{r}{\gamma}-c).
\]
Similarly, we have
\[
f(C^-(\eta,x,r))\sub C^-(\widehat f(\eta), f(x), \frac{r}{\gamma}-c)
\]
for non-trivial sets $C^-(\eta,x,r)$.
Furthermore, every boundary point that lives in $C^+(\eta,x,r)$ or $C^-(\eta,x,r)$ is mapped by $\widehat f$ to a boundary point that lives in $C^+(\widehat f(\eta), f(x), \frac{r}{\gamma}-c)$ or $C^-(\widehat f(\eta), f(x), \frac{r}{\gamma}-c)$, respectively.
Thus, $\widehat f$ is continuous with respect to both topologies.

Since $f$ is a \qiy, there exists a \qiy\ $g\colon X_2\to X_1$.
Let $\widehat g$ be the bijection $\rand X_2\to\rand X_1$ that is canonically defined by~$g$, which exists due to Theorem~\ref{thm_QIPreserveBoundaryDelta1}.
It is easy to see that $\widehat g\circ \widehat f$ is the identity on~$\rand X_1$.
So we have $\widehat g=\widehat f\inv$.
As $\widehat g$ is continuous with respect to both topologies, $\widehat f$ is a homeomorphism with respect to both topologies.
\end{proof}

\section{A pseudo-\hm\ for $X\cup\rand X$}\label{sec_BoundaryHM}

A subset $S$ of a \hm\ space $X$ is a \emph{base} of~$X$ if for every $x\in X$ there exists $s\in S$ with $d^\lrarrow(x,s)<\infty$.
We call $X$ \emph{finitely based} if it has a finite base.
For the following definition, recall that we set $R(i):=R(-i)$ if $R$ is an anti-ray and $i>0$.

Let $(X,d)$ be a finitely based hyperbolic geodesic \hm\ space that satisfies (\ref{itm_Bounded1}) and~(\ref{itm_Bounded2}).
Let $S$ be a finite base of~$X$.
Let $\eta,\mu\in X\cup\rand X$ and let $s\in S$.
If $\eta,\mu\in\rand X$, set
\[
\rho_s(\eta,\mu):=\sup_{R\in\eta,Q\in\mu}\{\liminf\{d^\lrarrow(s,P)\mid i,j\to\infty,\ P \text{ is an }R(i)\text{-}Q(j)\text{ geodesic}\}\}.
\]
If $\eta\in X$ and $\mu\in\rand X$, set
\[
\rho_s(\eta,\mu):=\sup_{Q\in\mu}\{\liminf\{d^\lrarrow(s,P)\mid i\to\infty,\ P \text{ is an }\eta\text{-}Q(i)\text{ geodesic}\}\}
\]
and
\[
\rho_s(\mu,\eta):=\sup_{Q\in\mu}\{\liminf\{d^\lrarrow(s,P)\mid i\to\infty,\ P \text{ is a }Q(i)\text{-}\eta\text{ geodesic}\}\}.
\]
If $\eta,\mu\in X$, set
\[
\rho_s(\eta,\mu):=\liminf\{d^\lrarrow(s,P)\mid P \text{ is an }\eta\text{-}\mu\text{ geodesic}\}.
\]
Set
\[
\rho_S(\eta,\mu):=\min\{\rho_s(\eta,\mu)\mid s\in S\}
\]
for all $\eta,\mu\in X\cup\rand X$.

We will see later that that we may have $\rho_S(\eta,\mu)=\infty$ for distinct $\eta,\mu\in\rand X$.
If we now follow the usual approach for hyperbolic geodesic metric spaces, we would define $\rho_S^\varepsilon(\eta,\mu):=e^{-\varepsilon\rho_S(\eta,\mu)}
$.
However, this is then undefined if $\rho_S(\eta,\mu)=\infty$.
Thus, we need the following slightly different definition in that we switch the supremum with the exponent.

Let $\varepsilon>0$.
For all $\eta,\mu\in\rand X$, we set
\[
\rho_s^\varepsilon(\eta,\mu):=\inf_{R\in\eta,Q\in\mu} \{e^{-\varepsilon\liminf\{d^\lrarrow(s,P)\mid i,j\to\infty,\ P \text{ is an }R(i)\text{-}Q(j)\text{ geodesic}}\}.
\]
For all $\eta\in X$ and $\mu\in\rand X$, we set
\[
\rho_s^\varepsilon(\eta,\mu):=\inf_{Q\in\mu} \{e^{-\varepsilon\liminf\{d^\lrarrow(s,P)\mid i\to\infty,\ P \text{ is an }\eta\text{-}Q(i)\text{ geodesic}}\}
\]
and
\[
\rho_s^\varepsilon(\eta,\mu):=\inf_{Q\in\mu} \{e^{-\varepsilon\liminf\{d^\lrarrow(s,P)\mid i\to\infty,\ P \text{ is a }Q(i)\text{-}\eta\text{ geodesic}}\}.
\]
For all $\eta,\mu\in X$, we set
\[
\rho_s^\varepsilon(\eta,\mu):=\inf_{R\in\eta,Q\in\mu} \{e^{-\varepsilon\liminf\{d^\lrarrow(s,P)\mid P \text{ is an }\eta\text{-}\mu\text{ geodesic}}\}.
\]
Set
\[
\rho_S^\varepsilon(\eta,\mu):=\max\{\rho_s^\varepsilon(\eta,\mu)\mid s\in S\}
\]
for all $\eta,\mu\in X\cup\rand X$.

Finally, we set
\[
d_{S,\varepsilon}(\eta,\mu):=\inf\{\sum_{i=0}^{n-1}\rho_S^\varepsilon(\eta_i,\eta_{i+1})\mid n\in\N, \eta=\eta_0,\eta_1,\ldots,\eta_n=\mu\in X\cup\rand X\}.
\]
for all $\eta,\mu\in X\cup\rand X$.

\begin{lem}\label{lem_BoundRho}
Let $\delta\geq 0$ and let $X$ be a $\delta$-hyperbolic geodesic \hm\ space with finite base $S\sub X$ that satisfies (\ref{itm_Bounded1}) and~(\ref{itm_Bounded2}) for the function $f\colon\R\to\R$.
Let $\varepsilon>0$ and set $\varepsilon':=e^{2\varepsilon(6\delta+2\delta f(\delta+1))}$.
Then
\[
\rho_S^\varepsilon(\eta_1,\eta_2)\leq \varepsilon'\max\{\rho_S^\varepsilon(\eta_1,\eta_3),\rho_S^\varepsilon(\eta_3,\eta_2)\}
\]
holds for all $\eta_1,\eta_2,\eta_3\in X\cup\rand X$.
\end{lem}

\begin{proof}
We assume that $\eta_1,\eta_2,\eta_3\in\rand X$: the case that some of them lie in~$X$ is dealt with in the same way but a bit simpler.
Let $R_1\in\eta_1$, $R_2\in\eta_2$ and $R_3,R_4\in\eta_3$.
Let $s\in S$ such that $\rho_s^\varepsilon(\eta_1,\eta_3)=\rho_S^\varepsilon(\eta_1,\eta_3)$.
Set $c:=6\delta+2\delta f(\delta+1)$.
We consider four points $R_1(i)$, $R_2(j)$, $R_3(i')$ and $R_4(j')$ for $i,j,i',j'\in\R$ such that
\begin{align*}
d(R_1(i),R_3(i'))&<\infty,\\
d(R_3(i'),R_4(j'))&<\infty\text{ and }\\
d(R_4(j'),R_2(j))&<\infty.
\end{align*}
Let $P_{12}$ be an $R_1(i)$-$R_2(j)$ geodesic, $P_{13}$ be an $R_1(i)$-$R_3(i')$ geodesic, $P_{14}$ be an $R_1(i)$-$R_4(j')$ geodesic, $P_{34}$ be an $R_3(i')$-$R_4(j')$ geodesic and $P_{42}$ be an $R_4(j')$-$R_2(j)$ geodesic.
Applying Lemma~\ref{lem_parallelSideCloseToAndFrom} twice, once to a geodesic triangle with sides $P_{12}$, $P_{14}$ and $P_{42}$ and once to a geodesic triangle with sides $P_{13}$, $P_{34}$ and $P_{14}$, we obtain that $P_{12}$ lies in the out- and in-ball of radius $2c$ around $P_{13}\cup P_{34}\cup P_{42}$.
Thus, we obtain
\[
d^\lrarrow(s,P_{12})+2c\geq \min\{d^\lrarrow(s,P_{13}),d^\lrarrow(s,P_{34}),d^\lrarrow(s,P_{42})\}.
\]

Since $R_3$ and~$R_4$ both lie in~$\eta_3$, we have 
\[
\liminf\{d^\lrarrow(s,P)\mid k,\ell\to\infty, P \text{ is an }R_3(k)\text{-}R_4(\ell)\text{ geodesic}\}=\infty.
\]
Thus, we obtain
\begin{align*}
\rho_s(\eta_1,\eta_2)+2c&\geq \min\{\rho_s(\eta_1,\eta_3),\rho_s(\eta_3,\eta_3),\rho_s(\eta_3,\eta_2)\}\\
&=\min\{\rho_s(\eta_1,\eta_3),\rho_s(\eta_3,\eta_2)\}
\end{align*}
and hence
\begin{align*}
\rho_S(\eta_1,\eta_2)+2c&=\rho_s(\eta_1,\eta_2)+2c\\
&\geq\min\{\rho_s(\eta_1,\eta_3),\rho_s(\eta_3,\eta_2)\}\\
&\geq\min\{\rho_S(\eta_1,\eta_3),\rho_S(\eta_3,\eta_2)\}.
\end{align*}
The assertion now follows from the definition of $\rho_S^\varepsilon$.
\end{proof}

A (pseudo-)\hm\ $d_a$ on $X\cup \rand X$ is a \emph{visual (pseudo-)\hm} with parameter $a>1$ if there is $C>0$ such that
\[
\frac{1}{C}a^{-\rho_S(\eta,\mu)}\leq d_a(\eta,\mu)\leq Ca^{-\rho_S(\eta,\mu)}
\]
for all $\eta,\mu\in X\cup\rand X$.

Now we prove that $d_{S,\varepsilon}$ is a visual pseudo-semimetric.
This proof is almost verbatim the same as for the case of metric spaces as in e.\,g.\ \cite[Proposition III.H.3.21]{BridsonHaefliger}.

\begin{thm}\label{thm_HMOnBoundary}
Let $\delta\geq 0$ and let $X$ be a $\delta$-hyperbolic geodesic \hm\ space with finite base $S\sub X$ that satisfies (\ref{itm_Bounded1}) and~(\ref{itm_Bounded2}).
Let $\varepsilon>0$ such that $\varepsilon'<\sqrt 2$ holds for $\varepsilon':=e^{2\varepsilon(6\delta+2\delta f(\delta+1))}$.
Then $d_{S,\varepsilon}$ is a visual pseudo-\hm\ on $X\cup \rand X$ that satisfies
\[
(3-2\varepsilon')\rho_S^\varepsilon(\eta,\mu)\leq d_{S,\varepsilon}(\eta,\mu)\leq\rho_S^\varepsilon(\eta,\mu)
\]
for all $\eta,\mu\in X\cup\rand X$.
\end{thm}

\begin{proof}
First, we note that the inequality $d_{S,\varepsilon}(\eta,\mu)\leq\rho_S^\varepsilon(\eta,\mu)$ holds trivially for all $\eta,\mu\in X\cup\rand X$ and that the definition of $d_{S,\varepsilon}$ implies directly that it is a pseudo-\hm.
We prove by induction on the length of chains $(\eta_0,\ldots,\eta_n)$ that
\begin{equation}\label{eq_HMOnBoundary_1}
(3-2\varepsilon')\rho_S^\varepsilon(\eta_0,\eta_n)\leq\sum_{i=0}^{n-1}\rho_S^\varepsilon(\eta_i,\eta_{i+1}).
\end{equation}
We set
\[
S(m):=\sum_{i=0}^{m-1}\rho_S^\varepsilon(\eta_i,\eta_{i+1}).
\]
Since $\varepsilon'>1$, we note that (\ref{eq_HMOnBoundary_1}) holds trivially if $n=1$ or $S(n)\geq 3-2\varepsilon'$.
So let us assume that $n\geq 2$ and $S(n)< 3-2\varepsilon'$.
Let $m\leq n$ be largest such that $S(m)\leq S(n)/2$.
Then we have
\[
\sum_{i=m+1}^{n-1}\rho_S^\varepsilon(\eta_i,\eta_{i+1})=S(n)-S(m+1)< S(n)/2.
\]
By induction we have
\[
\rho_S^\varepsilon(\eta_0,\eta_m)\leq \frac{S(n)}{2(3-2\varepsilon')}\text{\qquad and \qquad}\rho_S^\varepsilon(\eta_{m+1},\eta_n)\leq \frac{S(n)}{2(3-2\varepsilon')}.
\]
Trivially, we also have $\rho_S^\varepsilon(\eta_m,\eta_{m+1})\leq S(n)$.
Lemma~\ref{lem_BoundRho} applied twice implies
\begin{align*}
\rho_S^\varepsilon(\eta_0,\eta_n)&\leq(\varepsilon')^2\max\{\rho_S^\varepsilon(\eta_0,\eta_m),\rho_S^\varepsilon(\eta_m,\eta_{m+1}),\rho_S^\varepsilon(\eta_{m+1},\eta_n)\}\\
&\leq (\varepsilon')^2 S(n)\max\left\{1,\frac{1}{2(3-2\varepsilon')}\right\}.
\end{align*}
Since $(\varepsilon')^2/2\leq 1$ and $(\varepsilon')^2(3-2\varepsilon')\leq 1$ holds for all $1\leq \varepsilon'\leq \sqrt 2$, we immediately obtain~(\ref{eq_HMOnBoundary_1}).
Thus, we proved the inequality in the assertion and hence $d_{S,\varepsilon}$ is a visual pseudo-\hm.
\end{proof}

We note that the digraph of Example~\ref{ex_topBound} shows that, in the case of hyperbolicity, we cannot expect $d_{S,\varepsilon}$ to be a \hm.
However, we will prove that in the special case that $X$ is a one-ended locally finite digraph, we will prove that the pseudo-\hm\ is in fact a \hm\ (Proposition~\ref{prop_OneEndedBoundIsHM}).
We will look more closely to the situation that $d_{S,\varepsilon}(\eta,\mu)=0$ for distinct \qg\ boundary points $\eta$ and~$\mu$ in Proposition~\ref{prop_Dist=0}.

Let us now relate the topologies that we obtain from the pseudo-\hm\ $d_{S,\varepsilon}$ with the topologies from Section~\ref{sec_topology} in various situations.
For that, we call a subset $Y$ of a \hm\ space~$X$ \emph{independent} if $d(y,z)=\infty$ for all $y,z\in Z$.

\begin{prop}\label{prop_ToposOnBoundaryCoincide}
Let $\delta\geq 0$ and let $X$ be a $\delta$-hyperbolic geodesic \hm\ space with finite base $S\sub V(D)$ that satisfies (\ref{itm_Bounded1}) and~(\ref{itm_Bounded2}) such that either
\begin{enumerate}[\rm (i)]
\item\label{itm_ToposOnBoundaryCoincide_1} $|S|=1$ or
\item\label{itm_ToposOnBoundaryCoincide_2} for no $r\in\R$ and $x\in X$ the balls $\cB^+_r(x)$ and $\cB^-_r(x)$ contain an infinite independent point set.
\end{enumerate}
Let $\varepsilon>0$ such that $\varepsilon'<\sqrt 2$ holds for $\varepsilon':=e^{2\varepsilon(6\delta+2\delta f(\delta+1))}$.
Then the forward and backward topologies induced by $d_{S,\varepsilon}$ coincide with those of Section~\ref{sec_topology}.
\end{prop}

\begin{proof}
It suffices to prove that the neighbourhoods around \qg\ boundary points~$\omega$ coincide.
Furthermore, both types of topology can be treated analogously; so we just consider the forward topologies.

Let $f\colon\R\to\R$ such that $X$ satisfies (\ref{itm_Bounded1}) and~(\ref{itm_Bounded2}) for that function.
For $\gamma\geq 1$ and $c\geq 0$, let $\kappa(\delta,\gamma,c,f)\geq 0$ be the constant for geodesic stability, cp.\ Theorem~\ref{thm_Delta1GeodStab}.
Let $D> 0$ and let $r\geq 0$ such that
\[
\varepsilon r>\ln\frac{3-2\varepsilon'}{D}.
\]
Let $y\in C^+(\eta,s,r)$ for some $\eta\in\rand^+(\omega)$ and $s\in S$.
Then there is a \qg\ $R\in\omega$ such that for all $x$ on~$R$ there is an $x$-$y$ geodesic outside of $\cB^+_r(s)\cup\cB^-_r(s)$ and hence $\rho_s(\eta,y)\geq r$.
Now if we take the finite intersection
\[
C:=\bigcap\{C^+(\omega,s,r)\mid s\in S\},
\]
then we obtain $\rho_S(\omega,y)\geq r$ for all $y\in C$.
So Theorem~\ref{thm_HMOnBoundary} implies
\[
(3-2\varepsilon')\rho_S^\varepsilon(\eta,y)\leq(3-2\varepsilon')e^{-\varepsilon r}< D.
\]
Hence the out-ball of radius $D$ around~$\omega$ with respect to the pseudo-\hm\ $d_{S,\varepsilon}$ contains~$C$ and thus also
\[
\bigcap\{C^+_\rand(\omega,s,r)\mid s\in S\}.
\]
Since this is true for every $\omega$, this out-ball contains an f-neighbourhood with respect to the topology of Section~\ref{sec_topology}.

To show that every $C_\rand^+(\omega,x,r)$ such that $\omega$ contains an anti-ray contains some forward neighbourhood of~$\omega$ with respect to $d_{S,\varepsilon}$ it suffices to show that there exists some $k\in\R$ with
\[
\cB^+_r(x)\cup\cB^-_r(x)\sub\cB^+_k(S)\cup\cB^-_k(S).
\]
Let us suppose that this is false.
Since $S$ is a base, there is $s\in S$ with $d^\lrarrow(s,x)<\infty$.
We may assume $d(s,x)<\infty$; the other case follows with a symmetric argument.
So we have $\cB^+_r(x)\sub \cB^+_{r+d(s,x)}(s)$.
For every $i\in\N$, let $y_i\in\cB^-_r(x)$ with $d^\lrarrow(S,y_i)\geq i$.
By the pigeonhole principle, there is some $s'\in S$ with either $d(s',y_i)<\infty$ or $d(y_i,s)<\infty$ for infinitely many $i\in\N$.
If there were infinitely many $i\in \N$ with $d(s',y_i)<\infty$, then we can consider a geodesic triangle with $x$, $y_i$ and $s'$ as end points and apply Proposition~\ref{prop_reverseDistanceShortDelta1}\,(\ref{itm_reverseDistanceShortDelta1_1}) to bound $d(s',y_i)$ in terms of $d(y_i,x)$ and $d(s',x)$.
This contradicts $d^\lrarrow(S,y_i)\geq i$.
Hence, we may assume that $d(y_i,s)<\infty$ for infinitely many $i\in\N$.
Throwing all other out of the sequence, we may assume that $d(y_i,s)<\infty$ holds for all $i\in\N$.

In the case that (\ref{itm_ToposOnBoundaryCoincide_1}) holds, we directly obtain a contradiction to the choice of the points $y_i$: by applying Proposition~\ref{prop_reverseDistanceShortDelta1}\,(\ref{itm_reverseDistanceShortDelta1_1}) to geodesic triangles with $x$, $s$ and $y_i$ as end points we can bound $d(y_i,s)$ in terms of $d(s,x)$, $d(y_i,x)$ and~$f$, but we also know $d^\lrarrow(y_i,s)\leq d(y_i,s)\to\infty$ for $i\to\infty$.
This is not possible.
So let us now assume $|S|> 1$ and hence that (\ref{itm_ToposOnBoundaryCoincide_2}) holds.

Applying Proposition~\ref{prop_reverseDistanceShortDelta1}\,(\ref{itm_reverseDistanceShortDelta1_1}) to geodesic triangles with $x$, $y_i$ and $y_j$ as end points shows that $d^\lrarrow(y_i,y_j)$ is bounded by $2rf(\delta+1)$ if it is not~$\infty$.
If $d^\lrarrow(y_i,y_j)\neq\infty$, then a geodesic triangle with end points $s'$, $y_i$ and~$y_j$ shows that for fixed $i\in\N$ we obtain that $d(s',y_j)$ is bounded in terms of $d(s',y_i)$, $d^\lrarrow(y_i,y_j)$ and~$f$ by Proposition~\ref{prop_reverseDistanceShortDelta1}\,(\ref{itm_reverseDistanceShortDelta1_1}).
Thus, for fixed $i\in\N$ there are only finitely many $j\in\N$ with $d^\lrarrow(y_i,y_j)<\infty$.
Now we can easily find an infinite independent subset of $\{y_i\mid i\in\N\}$ which contradicts (\ref{itm_ToposOnBoundaryCoincide_2}).
\end{proof}

Note that it follows from Proposition~\ref{prop_ToposOnBoundaryCoincide} that the forward and backward topologies induced by $d_{S,\varepsilon}$ do not depend on the particular base for those \hm\ spaces satisfying the assumptions of that proposition.
These assumptions are satisfied by some natural classes of \hm\ spaces, e.\,g.\ hyperbolic digraphs of bounded degree and hence right cancellative hyperbolic semigroups, cp.\ Section~\ref{sec_semigroups}: they satisfy (\ref{itm_ToposOnBoundaryCoincide_2}) of Proposition~\ref{prop_ToposOnBoundaryCoincide}.
On the other side, hyperbolic monoids that have a finite generating set whose Cayley digraph satisfies (\ref{itm_Bounded2}) satisfy (\ref{itm_ToposOnBoundaryCoincide_1}) of Proposition~\ref{prop_ToposOnBoundaryCoincide}.

Let us briefly discuss the situation that the hyperbolic geodesic \hm\ space is not finitely based.
It is natural if we just replace in the definition of $\rho_S$ the minimum by the infimum over an infinite base.
Then it is easy to verify that $d_{S,\varepsilon}$ is a pseudo-\hm, too.
However, the topologies defined by $d_{S,\varepsilon}$ do not coincide with the topologies from Section~\ref{sec_topology} as the following example shows.

\begin{ex}
Let $G$ be a 3-regular tree and let $R=\ldots x_{-1}x_0x_1\ldots$ be a double ray in~$G$, i.\,e.\ all $x_i$ are different vertices and $x_i$ and $x_{i+1}$ are adjacent for all $i\in\Z$.
We orient the edges on~$R$ from $x_i$ to~$x_{i+1}$ and the other edges incident with vertices on~$R$ away from~$R$, i.\,e.\ if $y$ is a neighbour of~$x_i$, we orient the edge from~$x_i$ to~$y$.
All other edges will be replaced by two oppositely directed edges, i.\,e.\ an edge between $y,z$ that do not lie on~$R$ will be replaced by one edge directed from~$y$ to~$z$ and one edge directed from~$z$ to~$y$.
The resulting digraph $D$ is clearly hyperbolic and it is easy to see that it has no finite base.

The vertex set $S:=\{x_i\mid i\leq 0\}$ is an infinite base of~$D$.
Let us consider the \hm\ $d_{S,\varepsilon}$ where we replaced the definition of $\rho_S$ by
\[
\rho_S(\eta,\mu):=\inf\{\rho_s(\eta,\mu)\mid s\in S\}.
\]
Let $\omega\in\rand D$ such that $\ldots x_{-1}x_0\in\omega$.
Then it follows that $\rho_S(\omega,\eta)=0$ for all $\eta\in X\cup\rand X$ with $\eta\neq\omega$.
So we have $d_{S,\varepsilon}(\omega,\eta)=1$.
Thus, $\omega$ has a trivial f-neighbourhood, while the topology $\cO_f$ from Section~\ref{sec_topology} has no trivial neighbourhood of~$\omega$.
\end{ex}

\section{Properties of $X\cup\rand X$}\label{sec_PropBound}

In the situation of hyperbolic geodesic spaces, the boundaries are complete metric spaces and if the spaces are proper, the boundaries are compact, see e.\,g.\ \cite[Proposition III.H.3.7]{BridsonHaefliger}.
In the situation of \hm\ spaces, we get at least for digraph that local finiteness implies f- and b-completeness, see Theorem~\ref{thm_completeness}.
But before we proceed to that result, we first discuss situations with $d_{S,\varepsilon}(\eta,\mu)=0$ and then we prove the existence of geodesic rays, anti-rays and double rays with almost prescribed starting and end points.

We have already seen in Example~\ref{ex_topBound} that $d_{S,\varepsilon}(\eta,\mu)=0$ is possible, see the discussion after Theorem~\ref{thm_HMOnBoundary}.
Let us discuss this situation in a bit more detail.

\begin{prop}\label{prop_Dist=0}
Let $\delta\geq 0$ and let $X$ be a $\delta$-hyperbolic geodesic \hm\ space with finite base $S\sub X$ that satisfies (\ref{itm_Bounded1}) and~(\ref{itm_Bounded2}).
Let $\varepsilon>0$ such that $\varepsilon'<\sqrt 2$ holds for $\varepsilon':=e^{2\varepsilon(6\delta+2\delta f(\delta+1))}$.
Let $\eta,\mu,\omega\in\ X\cup\rand X$.
Then the following hold.
\begin{enumerate}[\rm(i)]
\item\label{itm_Dist=0_1} If $d_{S,\varepsilon}(\eta,\mu)=0$ and $\eta\neq\mu$, then $\eta,\mu\in\rand X$ and  $\eta\leq\mu$.
\item\label{itm_Dist=0_2} If $d_{S,\varepsilon}(\eta,\mu)=0=d_{S,\varepsilon}(\mu,\eta)$, then $\eta=\mu$.
\item\label{itm_Dist=0_2_a} If $d_{S,\varepsilon}(\eta,\mu)=0$ and $\eta\neq\mu$, then either $\eta$ contains only rays and $\mu$ contains only anti-rays or $\eta$ contains only anti-rays and $\mu$ contains only rays.
\item\label{itm_Dist=0_3} If $d_{S,\varepsilon}(\eta,\mu)=0$ and $d_{S,\varepsilon}(\mu,\omega)=0$, then we have either $\eta=\mu$ or $\mu=\omega$.
\end{enumerate}
\end{prop}

\begin{proof}
Let $\eta,\mu\in X\cup \rand X$ with $d_{S,\varepsilon}(\eta,\mu)=0$.
By definition of $d_{S,\varepsilon}$, it is only possible to have $d_{S,\varepsilon}(\eta,\mu)=0$ if $\eta$ and $\mu$ lie in~$\rand X$.
Let $R\in\eta$ and $Q\in\mu$.
They are $(\gamma,c)$-\qg s for some $\gamma\geq 1$ and $c\geq 0$.
Let $X$ satisfy (\ref{itm_Bounded1}) and~(\ref{itm_Bounded2}) for the function $f\colon\R\to\R$.
Let $\kappa$ be the constant for geodesic stability with respect to $\delta$, $\gamma$, $c$ and~$f$, cp.\ Theorem~\ref{thm_Delta1GeodStab}.
Let $P_1,P_2$ be two $R$-$Q$ geodesics with starting points $x_1,x_2$ and end points $y_1,y_2$, respectively.
Let $R_1$ be the directed subpath of~$R$ between $x_1$ and~$x_2$ and let $R_2$ be geodesic with the same starting and end points.
Let $Q_1$ be the directed subpath of~$Q$ between $y_1$ and~$y_2$ and let $Q_2$ be geodesic with the same starting and end points.
Let $i\neq j\in\{1,2\}$ such that $x_i$ is the starting point of~$R_1$ and $x_j$ is its end point.
Then there exists an $x_i$-$y_j$ geodesic~$P$.
By Lemma~\ref{lem_parallelSideCloseToAndFrom}, we know that $P$ lies in the out- and in-balls of radius $K:=6\delta+2\delta f(\delta+1)$ around $R_2\cup P_j$.

Let us first consider the case that $y_i$ is the starting point and $y_j$ is the end point of~$Q_2$.
Then $P$ lies in the out- and in-ball of radius~$K$ around $P_i\cup Q$ by Lemma~\ref{lem_parallelSideCloseToAndFrom}.
Since $d_{S,\varepsilon}(\eta,\mu)=0$, we may assume that we have chosen $P_2$ such that there is some point~$a$ on~$P$ with $d^\lrarrow (P_1,a)> K$ and $d^\lrarrow(P_2,a)>K$.
Applying geodesic stability, there exists $b_R$ on~$R_1$ and $b_Q$ on~$Q_1$ with $d(b_R,b_Q)\leq 2\kappa+2K$.
Since $d_{S,\varepsilon}(\eta,\mu)=0$, we may choose, for every $k\in\N$, the directed paths $R$, $Q$, $P_1$ and~$P_2$ such that $P_1$, $P_2$, $R_1$ and $Q_1$ lie outside of $\cB^+_k(S)\cup\cB^-_k(S)$.
Let $T_k$ be an $R_1$-$Q_1$ geodesic.
Let us suppose that there exists $\ell\in\N$ such that $\cB^+_\ell(x)\cup\cB^-_\ell(x)$ meets all~$T_k$.
Then either $\cB^+_{\ell+2K+2\kappa}(x)$ or $\cB^-_{\ell+2K+2\kappa}(x)$ contains a point of the directed subpaths $Q_1$ or~$R_1$ of~$Q$ or~$R$ that we used for the existence of~$T_k$, which is a contradiction to~(\ref{itm_Bounded1}) or~(\ref{itm_Bounded2}).
Thus, we have $R\leq Q$ and hence $\eta\leq \mu$ in this case.

Let us now consider the case that $y_j$ is the starting point and $y_i$ is the end point of~$Q_2$.
Then $Q_1$ lies in
\[
\cB^+_{K+\delta+\kappa}(R_2\cup P_j)\cup\cB^-_{\delta+\kappa}(P_1).
\]
So if we choose $P_2$ and $a$ on~$Q_1$ with $d^\lrarrow(P_1,a)\geq\kappa+\delta$ and $d(P_2,a)>K+\delta+\kappa$, then there is a directed $R_1$-$Q_1$ path of length at most $2\kappa+K+\delta$.
As in the previous case, we obtain $R\leq Q$ and $\eta\leq\mu$.
This finishes the proof of~(\ref{itm_Dist=0_1}).

Finally, (\ref{itm_Dist=0_2}) is an immediate consequence of~(\ref{itm_Dist=0_1}) and (\ref{itm_Dist=0_2_a}) and (\ref{itm_Dist=0_3}) directly follow from~(\ref{itm_Dist=0_1}) and Proposition~\ref{prop_BoundNo3Chain}.
\end{proof}

We note that there may be three distinct \qg\ boundary points $\eta,\mu,\omega$ with $d_{S,\varepsilon}(\omega,\eta)=0$ and $d_{S,\varepsilon}(\omega,\mu)=0$ as the following short example shows.

\begin{ex}\label{ex_Dist=0}
Let $D$ be the digraph that consists of a ray $R=x_0x_1\ldots$ and two anti-rays $Q_1=\ldots y_{-1}y_0$ and $Q_2\ldots z_{-1}z_0$ with edges from $x_i$ to $y_i$ and $z_i$ for all $i\in\N$.
The resulting digraph~$D$ is hyperbolic and $R$, $Q_1$ and $Q_2$ lies in distinct \qg\ boundary points $\omega$, $\eta$ and $\mu$, respectively.
The vertex $x_0$ is a base of~$D$.
Since $\omega\leq\eta$ and $\omega\leq\mu$, we have $d_{\{x_0\},\varepsilon}(\omega,\eta)=0$ and $d_{\{x_0\},\varepsilon}(\omega,\mu)=0$.
\end{ex}

In a digraph~$D$, we call $\ldots x_{-1}x_0x_1\ldots$ a \emph{double ray} if $x_ix_{i+1}$ is an edge of~$D$.
It is \emph{geodesic} if every finite directed subpath is geodesic.

\begin{prop}\label{prop_raysWithFixedEnds}
Let $\delta\geq 0$ and let $D$ be a $\delta$-hyperbolic digraph that satisfies (\ref{itm_Bounded1}) and (\ref{itm_Bounded2}) with finite base~$S$.
Let $\varepsilon>0$ such that $e^{2\varepsilon(6\delta+2\delta f(\delta+1))}<\sqrt 2$.
Then the following hold.
\begin{enumerate}[\rm (i)]
\item\label{itm_raysWithFixedEnds_1} If every vertex of~$D$ has finite out-degree, then for every $x\in V(D)$ and $\eta\in\rand D$ with $d_{S,\varepsilon}(x,\eta)<\infty$, then there is a geodesic ray~$R$ starting at~$x$ such that $R\in\mu$ for some $\mu\in\rand D$ with $\mu\leq\eta$.
\item\label{itm_raysWithFixedEnds_2} If every vertex of~$D$ has finite in-degree, then for every $x\in V(D)$ and $\eta\in\rand D$ with $d_{S,\varepsilon}(\eta,x)<\infty$, then there is a geodesic anti-ray~$R$ ending at~$x$ such that $R\in\mu$ for some $\mu\in\rand D$ with $\eta\leq\mu$.
\item\label{itm_raysWithFixedEnds_3} If $D$ is locally finite, then for all $\eta,\mu\in\rand D$  with $0<d_{S,\varepsilon}(\eta,\mu)<\infty$ there exists $\eta',\mu'\in\rand D$ with $\eta\leq\eta'$ and $\mu'\leq\mu$ such that there is a geodesic $\eta'$-$\mu'$ double ray.
\end{enumerate}
\end{prop}

\begin{proof}
Let us prove~(\ref{itm_raysWithFixedEnds_1}).
By definition of $d_{S,\varepsilon}$, there exists $Q\in\eta$ such that  for every $y$ on~$Q$ there is an $x$-$y$ geodesic in~$D$.
Let $Q=y_0y_1\ldots$ if~$Q$ is a ray and $Q=\ldots y_1y_0$ if $Q$ is an anti-ray.
Then there is a geodesic ray~$R$ starting at~$x$ such that each of its finite directed subpaths starting at~$x$ lie in infinitely many of these $x$-$y_i$ geodesics.
Using thin geodesic triangles with end vertices $x$, $y_0$ and $y_i$ we obtain for large~$i$ that almost all of the directed subpath of~$Q$ between $y_0$ and~$y_i$ lies in the out-ball of radius~$\delta$ around the $x$-$y_i$ geodesic.
Thus, all but a finite directed subpath of~$Q$ lies in the out-ball of radius~$\delta$ around~$R$.
This shows~(\ref{itm_raysWithFixedEnds_1}).

With a completely symmetric argument, we obtain~(\ref{itm_raysWithFixedEnds_2}).

To prove (\ref{itm_raysWithFixedEnds_3}), let $R\in\eta$ and $Q\in\mu$ such that either $R=x_0x_1\ldots$ or $R=\ldots x_1x_0$ and either $Q=y_0y_1\ldots$ or $Q=\ldots y_1y_0$.
By Theorem~\ref{thm_HMOnBoundary} and using hyperbolicity, there exists $M\geq 0$ such that for $i,j\to\infty$ all $x_i$-$y_j$ geodesics have a vertex of $\cB^+_M(S)\cup\cB^-_M(S)$.
Similar as before, we obtain a geodesic double ray~$P$ such that each of its inner directed subpaths lies in infinitely many of these $x_i$-$y_j$ geodesics with the property that neither the involved indices $i$ nor the involved indices~$j$ are bounded.
Let $P_1$ be an anti-ray and $P_2$ be a ray in~$P$.
Then hyperbolicity implies $R\leq P_1$ and $P_2\leq Q$, which implies the assertion.
\end{proof}

Now we are ready to prove the main result of this section.

\begin{thm}\label{thm_completeness}
Let $\delta\geq 0$ and let $D$ be a $\delta$-hyperbolic digraph that satisfies (\ref{itm_Bounded1}) and (\ref{itm_Bounded2}) for the function $f\colon\R\to\R$ with finite base~$S$.
Let $\varepsilon>0$ such that $e^{2\varepsilon(6\delta+2\delta f(\delta+1))}<\sqrt 2$.
Then the following hold.
\begin{enumerate}[\rm (i)]
\item\label{itm_completeness_1} If every vertex of~$D$ has finite out-degree, then $D\cup\rand D$ is sequentially f-compact.
\item\label{itm_completeness_2} If every vertex of~$D$ has finite in-degree, then $D\cup\rand D$ is sequentially b-compact.
\end{enumerate}
\end{thm}

\begin{proof}
Let every vertex of~$D$ have finite out-degree and let $(x_i)_{i\in\N}$ be a sequence in $V(D)\cup\rand D$ with $d_{S,\varepsilon}(x_i,x_j)<\infty$ for all $i<j$.
If $x_0\in\rand D$, then by definition of $d_{S,\varepsilon}$, there exists a vertex $x_0'$ on an element of~$x_0$ with $d_{S,\varepsilon}(x_0',x_1)<\infty$.
Thus, we have $d_{S,\varepsilon}(x_0',x_i)<\infty$ for all $i\in\N$ with $i\neq 0$.
Since b-convergence of sequences does not depend on the first element of the sequence, we may assume that $x_0$ is a vertex of~$D$.

For every $i\in\N$, if $x_i\in V(D)$, then let $P_i$ be an $x_0$-$x_i$ geodesic.
If $x_i\in \rand D$, then there exists a geodesic ray $P_i$ starting at~$x_0$ that lies in some $x_i'\in\rand D$ with $x_i'\leq x_i$ by Proposition~\ref{prop_raysWithFixedEnds}\,(\ref{itm_raysWithFixedEnds_1}).
These directed paths and rays define a geodesic ray~$R=v_0v_1\ldots$ such that infinitely many $P_i$ share the first edge of~$R$ among which infinitely many share the next edge of~$R$ and so on.
By switching to a subsequence of $(x_i)_{i\in\N}$, if necessary, we may assume that $P_i$ and $R$ have their first $i$ edges in common.
Let $\eta\in\rand D$ with $R\in\eta$.
We shall show that $(x_i)_{i\in\N}$ b-converges to~$\eta$.

For each $i\in\N$, let $u_i$ be the first vertex of~$P_i$ such that the next vertex on~$P_i$ does not lie in~$\cB^+_\delta(R)$, if it exists, and, if $P_i\sub\cB^+_\delta(R)$, let $u_i=x_i$ if $x_i\in V(D)$ and let $u_i$ be on $P_i$ with $d(x_0,u_i)\geq i$, otherwise.
Let $u_i'$ be on~$R$ with $d(u_i',u_i)\leq\delta$.
If $x_i$ is a vertex, set $y_i:=x_i$.
If $x_i\in\rand D$, then let $y_i$ be a vertex on $u_iP_ix_i$.
Then we have
\[
d_{S,\varepsilon}(y_i,x_j)\leq d_{S,\varepsilon}(y_i,x_i)+d_{S,\varepsilon}(x_i,x_j)<\infty
\]
for all $j>i$.
Hence there is a $y_i$-$x_j$ geodesic $P_{ij}$ if $x_j$ is a vertex.
If $x_j\in\rand D$, then there is a geodesic ray with starting vertex~$y_i$ that lies in a \qg\ boundary point $x_j'\leq x_j$ by Proposition~\ref{prop_raysWithFixedEnds} and hence we may assume that we have chosen $y_j$ such that there is a $y_i$-$y_j$ geodesic~$P_{ij}$.

Let $r>0$.
We consider the set $C^-_\rand(\eta,x_0,r)$.
Set
\begin{align*}
k&:=6\delta+2\delta f(\delta+1),\\
\ell_1&:=f(\delta+1)r+f(\delta),\\
\ell_2&:=(\ell_1+\delta)f(\delta+1)\\
\ell_3&:=\ell_2+k.
\end{align*}
Let $i\in\N$ such that $d(x_0,u_i)>\ell_3+\delta$ and set $m:=d(x_0,y_i)+3\delta+1$.
Let $j\in\N$ such that
\[
d(x_0,u_j)>d(x_0,y_i)+\delta+2
\]
and $x_0Rv_m$ lies in $\cB^+_\delta(P_j)$.
Note that this holds for all but finitely many $j$ by the choice of our sequence and the choices of~$u_j$ and~$y_j$.
Let $z_1$ be on~$P_{ij}$ with
\[
d(x_0,z_1)=d(x_0,y_i)+\delta+1.
\]
By hyperbolicity for the geodesic triangle with end vertices $x_0$, $y_i$ and $y_j$ and sides $P_iy_i$, $P_{ij}$ and $P_jy_j$, there exists a vertex $z_2$ on $P_j$ with $d(z_1,z_2)\leq\delta$.
By the choice of~$z_1$, of~$j$ and of~$m$, we know that $d(z_2P_j,v_m)\leq\delta$ and in particular $d(z_2,v_m)<\infty$.

Let $x$ on~$v_mR$.
Then there is a $y_i$-$x$ geodesic $P$.
Let $Q_1$ be a $u_i'$-$u_i$ geodesic and $Q_2$ a $u_i'$-$y_i$ geodesic.
We shall show that $P$ lies in $C^-_\rand(\eta,x_0,r)$.

Since $d(x_0,u_i)>\ell_3+\delta$, we have $d(x_0,Q_1)>\ell_3$.
If $d(x_0,Q_2)\leq \ell_2$, then the vertex verifying this distance lies in $\cB^-_k(Q_1\cup u_iP_iy_i)$ by Lemma~\ref{lem_parallelSideCloseToAndFrom} and thus, we find a vertex on $Q_1$ or $u_iP_iy_i$ that has distance at most $\ell_2+k=\ell_3$ from~$x_0$.
Since this is impossible, we have $d(x_0,Q_2)>\ell_2$.
Let $z$ be on $P$.
Then there exists a vertex $y$ either on~$Q_2$ with $d(y,z)\leq\delta$ or on $u_i'Rx$ with $d(z,y)\leq\delta$.
In the latter case, we directly obtain $d(x_0,z)>\ell_1$ and in the first case, we apply Proposition~\ref{prop_reverseDistanceShortDelta1}\,(\ref{itm_reverseDistanceShortDelta1_1}) and obtain $d(x_0,z)>\ell_1$ as well.
So $P$ lies outside of $\cB^+_{\ell_1}(x_0)$.
Finally, Proposition~\ref{prop_reverseDistanceShortDelta1}\,(\ref{itm_reverseDistanceShortDelta1_2}) implies $d(P,x_0)>r$ and thus we have shown that $P$ lies in $C^-_\rand(\eta,x_0,r)$.
This implies that $(x_i)_{i\in\N}$ b-converges to~$\eta$.

By an analogous argument, we obtain~(\ref{itm_completeness_2}).
\end{proof}

If $D\cup\rand D$ in Theorem~\ref{thm_completeness} is a \hm\ space, then we obtain the following corollary of Theorem~\ref{thm_completeness} by using Proposition~\ref{prop_CompactComplete}.

\begin{cor}\label{cor_completeness}
Let $\delta\geq 0$ and let $D$ be a $\delta$-hyperbolic digraph that satisfies (\ref{itm_Bounded1}) and (\ref{itm_Bounded2}) for the function $f\colon\R\to\R$ and that has a finite base~$S$.
Let $\varepsilon>0$ such that $e^{2\varepsilon(6\delta+2\delta f(\delta+1))}<\sqrt 2$.
Then the following hold.
\begin{enumerate}[\rm (i)]
\item If every vertex of~$D$ has finite out-degree and $D\cup\rand D$ is a \hm\ space, then $D\cup\rand D$ is f-complete.
\item If every vertex of~$D$ has finite in-degree and $D\cup\rand D$ is a \hm\ space, then $D\cup\rand D$ is b-complete.\qed
\end{enumerate}
\end{cor}

\section{The size of the boundary}\label{sec_size}

In the case of hyperbolic spaces, those hyperbolic boundary points that belong to a common end of the space form a connected set, see e.\,g.\ \cite[Proposition 7.5.17]{GhHaSur}, which immediately implies that an end with at least two hyperbolic boundary points contains continuum many of them.
We restrict ourselves to the case of digraphs here, since we defined ends only for them and not for general \hm\ spaces.
We cannot hope to prove that the \qg\ boundary is connected, since our two topologies make if very hard to ask for this: even a digraph is far from being connected in the topological sense, since e.\,g.\ a digraph on two vertices with a unique edge has the following partition into two open sets: one set is the end vertex of the edge and the other set consists of the starting vertex of the edge and the inner points of the edge.
Thus, we consider a different notion in our situation, which we will call semiconnectednes and that basically asks that the sets shall satisfy the connectedness condition with respect to both topologies simultaneously; see below for details.

As a first step in understanding the relations between the geodesic boundary and the ends better, we prove the following result.

\begin{prop}\label{prop_ComparableEndsImpliesFiniteDistance}
Let $\delta\geq 0$ and let $D$ be a locally finite $\delta$-hyperbolic digraph that satisfies (\ref{itm_Bounded1}) and (\ref{itm_Bounded2}) and that has a finite base~$S$.
Let $\varepsilon>0$ such that $\varepsilon'<\sqrt 2$ holds for $\varepsilon':=e^{2\varepsilon(6\delta+2\delta f(\delta+1))}$.
Let $\omega_1$ and $\omega_2$ be ends of~$D$ with $\omega_1\preccurlyeq \omega_2$ and let $\eta,\mu\in\rand D$ with $\eta\sub\omega_1$ and $\mu\sub\omega_2$.
Then we have $d_{S,\varepsilon}(\eta,\mu)<\infty$.

In particular, if $\omega_1=\omega_2$, then we have $d_{S,\varepsilon}(\eta,\mu)<\infty$ and $d_{S,\varepsilon}(\mu,\eta)<\infty$.
\end{prop}

\begin{proof}
Let $R\in\eta$ and $Q\in\mu$.
For every $n\in\N$ there exists a directed $R$-$Q$ path $P_n$ outside of $\cB^+_n(S)\cup\cB^-_n(S)$ by Proposition~\ref{prop_equivDefEnds}.
Let $P_n'$ be a geodesic with the same starting and end vertex as $P_n$.
It follows from the definition of $d_{S,\varepsilon}$ that $d_{S,\varepsilon}(\eta,\mu)<\infty$.

The additional statement follows trivially, since $\omega_1=\omega_2$ implies $\omega_2\preccurlyeq\omega_1$.
\end{proof}

We call a pseudo-\hm\ space $X$ \emph{semiconnected} if there is no partition $\{U,V\}$ of~$X$ such that $U$ and~$V$ are open with respect to~$\cO_b$ and~$\cO_f$.
A \emph{semiconnected component} is a maximal semiconnected subset of~$X$.
It is easy to see that distinct semiconnected components are disjoint and hence the semiconnected components form a partition of the pseudo-\hm\ space.

\begin{lem}\label{lem_distClosedSetsIsPositive}
Let $\delta\geq 0$ and let $D$ be a locally finite $\delta$-hyperbolic digraph that satisfies (\ref{itm_Bounded1}) and (\ref{itm_Bounded2}) and that has a finite base~$S$.
Let $\varepsilon>0$ such that $\varepsilon'<\sqrt 2$ holds for $\varepsilon':=e^{2\varepsilon(6\delta+2\delta f(\delta+1))}$.
Let $\varphi\colon \rand D\to\Omega D$ be the canonical map with $\eta\sub\varphi(\eta)$ for all $\eta\in\rand D$.
Let $\omega\in\Omega D$ and let $A$, $B$ be two subsets of~$\rand D$ with $\varphi\inv(\omega)\sub A\cup B$ and $A\cap B\cap\varphi\inv(\omega)=\es$ such that $A$ is closed in~$\cO_f$ and $B$ is closed in~$\cO_b$.
If $(D\cup\rand D,d_{S,\varepsilon})$ is a \hm\ space, then
\[
d_{S,\varepsilon}(A\cap\varphi\inv(\omega),B\cap\varphi\inv(\omega))>0.
\]
\end{lem}

\begin{proof}
Let us suppose that $d_{S,\varepsilon}(A\cap\varphi\inv(\omega),B\cap\varphi\inv(\omega))=0$.
By Proposition~\ref{prop_ComparableEndsImpliesFiniteDistance} we know that $d_{S,\varepsilon}(a_i,a_j)<\infty$ and $d_{S,\varepsilon}(b_i,b_j)<\infty$ for all $i,j\in\N$.
Then there are sequences $(a_i)_{i\in\N}$, $(b_i)_{i\in\N}$ in $A\cap\varphi\inv(\omega)$, in $B\cap\varphi\inv(\omega)$, respectively, such that $d_{S,\varepsilon}(a_i,b_i)\to 0$ for $i\to\infty$.
By Theorem~\ref{thm_completeness}, there exists $a\in\rand D$ such that $(a_i)_{i\in\N}$ has a subsequence that f-converges to~$a$.
By replacing $(a_i)_{i\in\N}$ by this subsequence, we may assume that $(a_i)_{i\in\N}$ f-converges to~$a$.
But then we also replace $(b_i)_{i\in\N}$ by a subsequence such that $d_{S,\varepsilon}(a_i,b_i)\to 0$ for $i\to\infty$ is still satisfied.
Applying Theorem~\ref{thm_completeness} once more, there is a subsequence of $(b_i){i\in\N}$ that b-converges to some $b\in\rand D$.
Again, we switch to subsequences to obtain that $(b_i)_{i\in\N}$ b-converges to~$b$ and that $d_{S,\varepsilon}(a_i,b_i)\to 0$ for $i\to\infty$.
Thus, we obtain
\[
d_{S,\varepsilon}(a,b)\leq d_{S,\varepsilon}(a,a_i)+d_{S,\varepsilon}(a_i,b_i)+d_{S,\varepsilon}(b_i,b)\to 0 \text{ for }i\to\infty.
\]

Since $d_{S,\varepsilon}$ is a \hm, we conclude $a=b$.
We have $a\in A$ and $b\in B$ since $A$ is closed in~$\cO_f$ and~$B$ is closed in ~$\cO_b$.
Thus, $A\cap B$ is not empty, which contradicts the assumptions.
Thus, the assertion follows.
\end{proof}

\begin{thm}\label{thm_SemiConComBound}
Let $\delta\geq 0$ and let $D$ be a locally finite $\delta$-hyperbolic digraph that satisfies (\ref{itm_Bounded1}) and~(\ref{itm_Bounded2}) and that has a finite base~$S$.
Let $\varepsilon>0$ such that $\varepsilon'<\sqrt 2$ holds for $\varepsilon':=e^{2\varepsilon(6\delta+2\delta f(\delta+1))}$.
Let $\varphi\colon \rand D\to\Omega D$ be the canonical map with $\eta\sub\varphi(\eta)$ for all $\eta\in\rand D$.
If $(D\cup\rand D,d_{S,\varepsilon})$ is a \hm\ space, then there is, for every $\omega\in\Omega_D$, a unique semiconnected component containing $\varphi\inv(\omega)$.
\end{thm}

\begin{proof}
Let $\omega\in\Omega D$ and suppose that $\varphi\inv(\omega)$ does not lie in a unique semiconnected component.
Let $X\sub\rand D$ be the union of all semiconnected components that meet~$\varphi\inv(\omega)$.
Since $X$ is not semiconnected, there is a partition $\{A',B'\}$ of~$X$ such that both $A'$ and~$B'$ are open in $\cO_f$ and in~$\cO_b$.
Then their complements $A$ and $B$ in $\rand D$ are closed with respect to both topologies, cover $\varphi\inv(\omega)$ and are disjoint inside $\varphi\inv(\omega)$.
Thus, we can apply Lemma~\ref{lem_distClosedSetsIsPositive} and obtain
\begin{equation}\label{itm_SemiConCompBound_1}
d_{S,\varepsilon}(A\cap\varphi\inv(\omega),B\cap\varphi\inv(\omega))>0.
\end{equation}

Let $\eta\in A\cap\varphi\inv(\omega)$ and $\mu\in B\cap\varphi\inv(\omega)$.
Since $d_{S,\varepsilon}$ is a \hm, Proposition~\ref{prop_raysWithFixedEnds}\,(\ref{itm_raysWithFixedEnds_3}) implies the existence of a geodesic $\eta$-$\mu$ double ray~$R$.
Let $R_1$ be an anti-subray of~$R$ that lies in~$\eta$ and let $R_2$ be a subray of~$R$ that lies in~$\mu$.
Since $\eta$ and $\mu$ belong to the same end, there exists, for every $n\in\N$, a directed $R_1$-$R_2$ path $P_n$ that lies outside of $\cB^+_n(S)\cup\cB^-_n(S)$.
For every $n>N$ there exists a vertex $x_n$ on~$P_n$ that lies outside of~$A'$ and~$B'$.
Moreover, we may choose $x_n$ such that
\begin{align*}
d_{S,\varepsilon}(A\cap\varphi\inv(\omega),x_n)&>d_{S,\varepsilon}(A\cap\varphi\inv(\omega),B\cap\varphi\inv(\omega))\text{ and}\\
d_{S,\varepsilon}(x_n,B\cap\varphi\inv(\omega))&>d_{S,\varepsilon}(A\cap\varphi\inv(\omega),B\cap\varphi\inv(\omega)).
\end{align*}

Since there are directed paths from $R_1$ to~$x_i$ and from $x_i$ to~$R_2$ for every $i\in\N$, we have $d_{S,\varepsilon}(\eta,x_i)<\infty$ and $d_{S,\varepsilon}(x_i,\mu)<\infty$.
Thus, Proposition~\ref{prop_raysWithFixedEnds}\,(\ref{itm_raysWithFixedEnds_1}) and (\ref{itm_raysWithFixedEnds_2}) imply the existence of geodesic $\eta$-$x_i$ anti-rays $Q^i_1$ and geodesic $x_i$-$\mu$ rays $Q^i_2$.
Since $d_{S,\varepsilon}$ is a visual \hm\ by Theorem~\ref{thm_HMOnBoundary}, since $D$ is locally finite and by the choices of the~$x_n$, there exists a vertex that lies on infinitely many of these anti-rays $Q^i_1$ and a vertex that lies on infinitely many of these rays $Q^i_2$.
Hence, there exists a geodesic double ray $Q_1$ and a subset $I\sub\N$ such that some anti-subray of~$Q_1$ lies in all $Q_1^i$ for $i\in\N$ and every other vertex lies on all but finitely many of the anti-rays $Q_1^i$ for $i\in I$.
Let $\nu_1\in\rand D$ such that some subray of~$Q_1$ lies in~$\nu_1$.
By changing the sequence  $(x_i)_{i\in\N}$, we may assume that $I=\N$.
Analogously, we use the geodesic rays $Q_2^i$ to define a geodesic double ray $Q_2$ with similar properties as $Q_1$ that goes from some $\nu_2\in\rand D$ to~$\mu$.

Let $P$ be a $Q_1$-$Q_2$ geodesic.
Let $x$ be its starting vertex and $x'$ be its end vertex.
For every large enough $i\in\N$, we consider the geodesic triangle with end vertices $x, x', x_i$ and sides $xQ_1^ix_i$, $x_iQ_2^ix'$ and~$P$.
Since there are only finitely many vertices close to or from~$P$, we find infinitely many disjoint directed $Q_1$-$Q_2$ paths of length at most~$\delta$.
Thus, we have $d_{S,\varepsilon}(\nu_1,\nu_2)=0$ and since $d_{S,\varepsilon}$ is a \hm, we conclude $\nu_1=\nu_2$.

Let $Q_1^+$ be a subray of~$Q_1$ and let $Q_2^-$ be an anti-subray of~$Q_2$.
For every $n\in\N$ all but finitely many $x_i$ have a geodesic $T_i$ from~$Q_1^+$ to them outside of $\cB^+_n(S)\cup\cB^-_n(S)$.
Thus, for all but finitely many $i\in\N$, the composition of~$T_i$ with $x_iP_i$ is a directed path from $Q_1^+$ to~$R_2$ outside of $\cB^+_n(S)\cup\cB^-_n(S)$.
Similarly, we find for every $n\in\N$ a directed path from~$R_1$ to~$Q_2^-$ outside of $\cB^+_n(S)\cup\cB^-_n(S)$.
Since $R_1$ and~$R_2$ belong to the same end~$\omega$, we conclude that $\nu_1$ lies in $\varphi\inv(\omega)$, too.

Since $Q_1^+\leq Q_2^-\leq Q_1^+$, the sequence $(x_i)_{i\in\N}$ f-converges to~$\nu_1$: we find for every $n\in\N$ for all but finitely many $i\in\N$ a directed path from $x_i$ to~$Q_1^+$ outside of $\cB^+_n(S)\cup\cB^-_n(S)$.
Similarly, $(x_i)_{i\in\N}$ b-converges to~$\nu_2$.
So if $\nu_1$ were in either $A$ or~$B$, some subsequence of $(x_i)_{i\in\N}$ must lie in either $A'$ or~$B'$.
Since this is false by the choice of the vertices $x_i$, we conclude that $\nu$ lies in neither $A$ nor~$B$.
This is a contradiction to the fact that $A\cup B$ covers $\varphi\inv(\omega)$.
\end{proof}

In order to find the preimages of semiconnected components of~$\rand D$ in~$\Omega D$, we pose Problem~\ref{prob_preimageSemiconComp}, for which we need the following definition.

The \emph{components} of an order $(X,\leq)$ are the maximal subsets $Y$ of~$X$ such that for all $x,y\in Y$ there are $z_1,\ldots z_n\in Y$ with $x=z_1$ and $y=z_n$ and such that either $z_i\leq z_{i+1}$ or $z_{i+1}\leq z_i$ for all $1\leq i<n$.

\begin{prob}\label{prob_preimageSemiconComp}
Let $D$ be a locally finite hyperbolic digraph that satisfies (\ref{itm_Bounded1}) and~(\ref{itm_Bounded2}) and that has a finite base and let $\varphi\colon \rand D\to\Omega D$ be the canonical map.
Are the semiconnected components of $\rand D$ the preimages under~$\varphi$ of the components of the ends with respect to~$\preccurlyeq$?
\end{prob}

Now we are turning our attention to the question of how large the geodesic boundary of locally finite digraphs can be.
In order to count the geodesic boundary points, we need the following result on the number of elements of semiconnected \hm\ spaces.

\begin{prop}\label{prop_ElementsSemiConComp}
Every semiconnected \hm\ space with at least two elements contains infinitely many elements.
\end{prop}

\begin{proof}
Let us suppose that a semiconnected \hm\ space $X$ has more than one but only finitely many elements.
Then any partition of its elements into two sets has the property that both of its sets are open in~$\cO_f$ and in~$\cO_b$, which is impossible.
\end{proof}

As another preliminary result, we prove that for one-ended locally finite hyperbolic digraphs in our usual setting, the pseudo-\hm\ is indeed a \hm.

\begin{prop}\label{prop_OneEndedBoundIsHM}
Let $\delta\geq 0$ and let $D$ be a locally finite $\delta$-hyperbolic digraph that satisfies (\ref{itm_Bounded1}) and~(\ref{itm_Bounded2}) and that has a finite base~$S$.
Let $\varepsilon>0$ such that $\varepsilon'<\sqrt 2$ holds for $\varepsilon':=e^{2\varepsilon(6\delta+2\delta f(\delta+1))}$.
If there are $\eta,\mu\in\rand D$ with $d_{S,\varepsilon}(\eta,\mu)=0$ and $d_{S,\varepsilon}(\mu,\eta)<\infty$, then $\eta=\mu$.

In particular, if $D$ is one-ended, then $d_{S,\varepsilon}$ is a \hm.
\end{prop}

\begin{proof}
Let $\eta,\mu\in\rand D$ with $d_{S,\varepsilon}(\eta,\mu)=0$ and $d_{S,\varepsilon}(\mu,\eta)<\infty$ and let us suppose that $\eta\neq\mu$.
By Proposition~\ref{prop_raysWithFixedEnds}\,(\ref{itm_raysWithFixedEnds_3}) there exists a geodesic double ray $R$ from $\mu'$ to~$\eta'$, where $\mu',\eta'\in\rand D$ with $d(\mu,\mu')=0$ and $d(\eta',\eta)=0$.
Proposition~\ref{prop_Dist=0}\,(\ref{itm_Dist=0_3}) implies $\eta'=\eta$ and $\mu=\mu'$.
Let $R_1$ be a subray of~$R$ and let $R_2$ be a anti-subray of~$R$.
Since $d_{S,\varepsilon}(\eta,\mu)=0$, we have $R_1\leq R_2$ by Proposition~\ref{prop_Dist=0}\,(\ref{itm_Dist=0_1}).
Hence, there exists $M\in\N$ such that outside all balls $\cB^+_r(S)$ and $\cB^-_r(S)$ there exists an $R_1$-$R_2$ geodesic of length at most~$M$.
This contradicts Proposition~\ref{prop_reverseDistanceShortDelta1}\,(\ref{itm_reverseDistanceShortDelta1_2}) as the reverse distance between the end vertices of these geodesics strictly increases.
Thus, we have $\eta=\mu$.

The additional statement immediately follows from Proposition~\ref{prop_ComparableEndsImpliesFiniteDistance}.
\end{proof}

\begin{cor}\label{cor_NumberBoundayPointsDigraph}
Let $D$ be a one-ended locally finite hyperbolic digraph that satisfies (\ref{itm_Bounded1}) and~(\ref{itm_Bounded2}) and that has a finite base.
Then $\rand D$ has either a unique or infinitely many elements.
\end{cor}

\begin{proof}
Let $\delta\geq 0$ such that $D$ is $\delta$-hyperbolic.
Let $S$ be a finite base of~$D$ and let $\varepsilon>0$ such that $\varepsilon'<\sqrt 2$ holds for $\varepsilon':=e^{\varepsilon(6\delta+2\delta f(\delta+1))}$.
Then Proposition~\ref{prop_OneEndedBoundIsHM} implies that $d_{S,\varepsilon}$ is a \hm.
So we can apply Theorem~\ref{thm_SemiConComBound} and obtain that $\rand D$ is semiconnected.
The assertion follows from Proposition~\ref{prop_ElementsSemiConComp}.
\end{proof}

\section{Hyperbolic semigroups}\label{sec_semigroups}

Let $S$ be a semigroups and let $A$ be a finite generating set of~$S$.
The \emph{(right) Cayley digraph} of~$S$ (with respect to~$A$) has $S$ as its vertex set and edges from $x$ to $xa$ for all $x\in S$ and $a\in A$.
This way, $S$ is a \hm\ space.
If $S$ is finitely generated, then we call it \emph{hyperbolic} if it has a finite generating set such that its Cayley digraph with respect to that generating set is hyperbolic.

A straight-forward argument shows that different finite generating sets define \qi\ Cayley digraphs, see Gray and Kambites \cite[Proposition 4]{GK-SemimetricSpaces} and thus, we obtain that the property of being hyperbolic does not depend on the particular generating set for right cancellative finitely generated semigroups, see~\cite[Proposition 8.1]{H-HyperbolicDigraph}.

Theorems~\ref{thm_QIPreserveBoundaryDelta1}, \ref{thm_QIPreserveGeodBoundaryDigraphs} and~\ref{thm_QisImplyHomeo} imply that the homeomorphism types of the (quasi-)\linebreak geodesic f-boundary of finitely generated semigroups whose Cayley digraphs satisfy (\ref{itm_Bounded2}) does not depend on the particular generating set and, if the semigroup is right cancellative, then the same holds for the (quasi-)geodesic boundary.
Thus, we denote by $\rand^f S$, for a finitely generated semigroup~$S$, the quasi-geodesic f-boundary of~$S$ and, if $S$ is right cancellative, we denote by $\rand S$ the quasi-geodesic boundary of~$S$.

The results of Section~\ref{sec_size} together with results on the number of ends of semigroups by Craik et al.\ \cite{CGKMR-EndsSemigroups} enable us to obtain some results on the size of the \qg\ boundary of hyperbolic semigroups.
First, we immediately have the following corollary of Corollary~\ref{cor_NumberBoundayPointsDigraph}.

\begin{cor}\label{cor_sizeBoundarySemigroup}
Let $S$ be a one-ended finitely generated right cancellative hyperbolic semigroup.
Then $\rand S$ has either exactly one or infinitely many elements.\qed
\end{cor}

The possible numbers of ends of left cancellative semigroups were determined by Craik et al.\ \cite[Theorem 3.7]{CGKMR-EndsSemigroups} to be in $\{0,1,2,\infty\}$.
This immediately implies the following corollary for cancellative semigroups.
(Note that if the geodesic boundary satisfies the separation axiom $T_1$ with respect to $\cO_f$ or $\cO_b$, then it does so for the other topology as well and it is a \hm\ space.)

\begin{cor}
Let $S$ be a finitely generated cancellative hyperbolic semigroup such that its geodesic boundary is a $T_1$ space with respect to either $\cO_f$ or~$\cO_b$.
Then $|\rand S|\in\{0,1,2,\infty\}$.\qed
\end{cor}

An obvious question arising is whether the assumption of $T_1$ separability is necessary.

\begin{prob}
Does there exist a finitely generated right cancellative hyperbolic semigroup whose geodesic boundary is not a \hm\ space?
\end{prob}

In the situation of hyperbolic groups, those with few boundary points, i.\,e.\ with at most two, are called elementary and their structure can be described pretty easily in that they are either finite or \qi\ to~$\Z$, cf.\ \cite[Theorem 2.28]{KB-BoundaryHypGroup}.
In analogy to the case of groups, we call a finitely generated right cancellative hyperbolic semigroup \emph{elementary} if it has at most two geodesic boundary points.

The right cancellative hyperbolic semigroups without geodesic boundary points have no end as well by Proposition~\ref{prop_boundaryIsRefinement}\,(\ref{itm_boundaryIsRefinement_3}).
So they are finite.

An example for a finitely generated right cancellative hyperbolic semigroup with a unique geodesic boundary point is~$\N$ and, in analogy to the group case, one might think that all other examples are \qi\ to~$\N$.
However, this is not the case as the following example shows.

\begin{ex}
Consider the monoid $S:=\left<a,b\mid a^2=b^2, ab=ba\right>$.
It is straight-forward to check that this is hyperbolic and that it has a unique geodesic boundary point.
To see that $S$ is not \qi\ to~$\N$, it suffices to note that $d(a,b)=\infty=d(b,a)$ but that there are no two elements of~$\N$ with such a property.
\end{ex}

Still, the monoid of the previous example has a structure that reminds very much of~$\N$ but the connection is weaker than \qiy.

Let us now consider the case of precisely two geodesic boundary points.
It follows directly from Corollary~\ref{cor_sizeBoundarySemigroup} that finitely generated right cancellative hyperbolic semigroups have exactly two ends.
If the semigroup is cancellative, then it is a two-ended group by Craik et al.\ \cite[Corollary 3.8]{CGKMR-EndsSemigroups}, so it is \qi\ to~$\Z$.
It remains open to look at the case of right cancellative semigroups that are not cancellative: do those still look like~$\Z$ in a way as in the case of exactly one geodesic boundary point the semigroups look like~$\N$?

\providecommand{\bysame}{\leavevmode\hbox to3em{\hrulefill}\thinspace}
\providecommand{\MR}{\relax\ifhmode\unskip\space\fi MR }
\providecommand{\MRhref}[2]{%
  \href{http://www.ams.org/mathscinet-getitem?mr=#1}{#2}
}
\providecommand{\href}[2]{#2}

\end{document}